\documentclass[12pt]{article}
\usepackage{geometry}                % See geometry.pdf to learn the layout options. There are lots.
\usepackage{graphicx,color}
\usepackage{amssymb,amsmath,amsthm,mathrsfs}
\usepackage[all,cmtip]{xy}
\usepackage{epstopdf, comment, esint, enumitem}

\usepackage[pdftex,bookmarks,pdfnewwindow,plainpages=false,unicode,pdfencoding=auto]{hyperref}

\usepackage[pdftex,bookmarks,pdfnewwindow,plainpages=false,unicode,pdfencoding=auto]{hyperref}

\newtheorem{theorem}{Theorem}[section]
\newtheorem{corollary}[theorem]{Corollary}
\newtheorem{lemma}[theorem]{Lemma}

\theoremstyle{remark}
\newtheorem*{remark}{Remark}

\numberwithin{equation}{section}

\DeclareMathOperator{\diam}{diam}
\DeclareMathOperator{\aut}{Aut}

\DeclareMathOperator{\hol}{Hol}
\DeclareMathOperator{\hyp}{hyp}

\DeclareMathOperator{\re}{Re }
\DeclareMathOperator{\im}{Im }

 \DeclareMathOperator{\lin}{lin }

 \DeclareMathOperator{\rough}{rough }
 
 \DeclareMathOperator{\thin}{thin }
 \DeclareMathOperator{\LS}{LS }

\begin{document}

\title{Dynamics of simply parabolic inner functions}

\author{Oleg Ivrii}

\date{June 17, 2026}

\maketitle

\begin{abstract}
We study the dynamics of Polya-Szeg\"o inner functions and discuss some of their basic properties such as equivalent conditions for simple and double parabolicity. We show that a simply parabolic Polya-Szeg\"o inner function admits forward and backward quotient half-cylinders, which allows one to enrich its dynamics with a Lavaurs map.

To proceed, we restrict our attention to simply parabolic inner functions with finite Lyapunov exponent: $\int_{\mathbb{R}} \log |F'| d\ell < \infty$. We define a geodesic flow on the Riemann surface lamination associated to the Lavaurs semigroup and show that it is ergodic. As an application, we establish the Orbit Counting Theorem up to a Ces\`aro average for Lavaurs semigroups. If we additionally assume that $F$ is a parabolic one component inner function, then the geodesic flow is mixing and the full Orbit Counting Theorem holds.
 \end{abstract}

\section{Introduction}

An {\em inner function} $f(z)$ is a  holomorphic self-map of the unit disk $\mathbb{D}$ such that for a.e.~$\theta \in [0, 2\pi)$, the radial boundary value $f(e^{i\theta}) := \lim_{r \to 1} f(re^{i\theta})$ exists and is unimodular. Throughout this paper, we assume that $f$ is neither M\"obius nor constant. Inner functions can be characterized into three types: hyperbolic, boundary hyperbolic and parabolic.

An inner function $f$ is {\em hyperbolic} if the Denjoy-Wolff fixed point $p$ lies in the unit disk. In the {\em boundary hyperbolic} case, the Denjoy-Wolff point lies on the unit circle and the angular derivative $f'(p) := \lim_{r \to 1} f(rp) < 1$, while in the {\em parabolic} case, the angular derivative $f'(p) = 1$.

The parabolic case can be split into two further sub-cases. Assume for a second that $f$ is analytic in a neighbourhood of $p$. If the Taylor expansion of $f$ centered at $p$ is
\begin{equation*}
\label{eq:1p-expansion}
z + c(z-p)^2 + \dots, \qquad c \ne 0,
\end{equation*}
then $f$ is {\em simply parabolic}\/. Otherwise, $f$ has the expansion
\begin{equation*}
\label{eq:2p-expansion}
f(z) = z + c(z-p)^3 + \dots,
\end{equation*}
and $f$ is {\em doubly parabolic}\/. We will momentarily explain how to distinguish between simple and double parabolicity when $f$ is not analytic in any neighbourhood of $p$.

We can view $f$ as a self-map of the upper half-plane $\mathbb{H}$ by conjugating it with the M\"obius transformation $m: \mathbb{D} \to \mathbb{H}$ which maps the unit disk to the upper half-plane and takes $p$ to $\infty$. In other words, we consider the map $F: \mathbb{H} \to \mathbb{H}$ given by $F = m \circ f \circ m^{-1}$. It is well known that the Lebesgue measure $\ell$ on the real line $\mathbb{R}$ is invariant under $F$.
The following chart presents several equivalent conditions, each of which may be taken as the definition of simple or double parabolicity:

\begin{table}[h!]
\centering
\begin{tabular}{|c|p{6cm}|p{6cm}|}
\hline
 & \textbf{Simply parabolic} & \textbf{Doubly parabolic} \\ \hline
1 & Positive hyperbolic step: \newline $\displaystyle \lim d_{\mathbb{H}}(z_n, z_{n+1}) > 0$ 
  & Zero hyperbolic step: \newline $\displaystyle \lim d_{\mathbb{H}}(z_n, z_{n+1}) = 0$ \\ \hline
2 & $\mathbb{H}/ \langle F \rangle \cong \mathbb{D}^*$ is a half-cylinder 
  & $\mathbb{H}/ \langle F \rangle \cong \mathbb{C}^*$ is a cylinder \\ \hline
3 & $\ell$ is not ergodic & $\ell$ is ergodic \\ \hline
4 & $F$ admits AC conjugacies & $F$ is rigid \\ \hline
\end{tabular}
\caption{The dichotomy between simply and doubly parabolic inner functions}
\label{tab:adm}
\end{table}

We only explain the first condition in detail. Let $(z_n)_{n=0}^\infty$ be a forward orbit. By the Schwarz lemma, the hyperbolic distance
$d_{\mathbb{H}}(z_n, z_{n+1})$ is decreasing. A parabolic inner function $F$ has {\em positive step} if for some (and hence every) forward orbit $(z_n)_{n=0}^\infty$,  
$\lim_{n \to \infty} d_{\mathbb{H}}(z_n, z_{n+1}) > 0$. Otherwise, $F$ has {\em zero step}. By \cite[Theorem G]{doering-mane}, $F$ has zero step if and only if
 $$d_{\mathbb{H}}(F^{\circ n}(z),F^{\circ n}(w)) \to 0, \qquad \text{for any }z, w \in \mathbb{H}.$$

The equivalence of the conditions (1) and (2) was established by C.~Cowen \cite{cowen} by constructing a fundamental set $V$ for a holomorphic self-map $F$ of the upper half-plane. The natural extension of $F|_V$ is conformally conjugate to one of the four standard model maps, with the simply parabolic case corresponding to $(\mathbb{H}, z \to z+1)$ and the doubly parabolic case corresponding to $(\mathbb{C}, z \to z+1)$.

The equivalence of (1) and (3) was established by J.~Aaronson \cite{aaronson81}. In fact, Aaronson proved the stronger statement that an inner function is ergodic if and only if it is exact.

The equivalence of (1) and (4) was established by M.~Shub and D.~Sullivan \cite{SS} for hyperbolic Blaschke products using techniques from one-dimensional dynamics, while the general case was proved later by D.~Hamilton \cite{hamilton-i} who employed quasiconformal deformations and a conjugacy result of Pommerenke \cite{pommerenke}.

For additional background on the classification of the parabolic types and related techniques, we refer the reader to the survey \cite{doering-mane}.

\subsection{Polya-Szeg\"o inner functions}

As explained in \cite{doering-mane}, a general parabolic inner function $F: \mathbb{H} \to \mathbb{H}$ has the form
\begin{equation}
\label{eq:general-inner}
F(z) = z + S - \int_{\mathbb{R}} \frac{1+az}{z - a} \, d\nu(a),
\end{equation}
where $S \in \mathbb{R}$ and $\nu \ge 0$ is a finite singular measure on the real line. We say that an inner function belongs to the {\em Polya-Szeg\"o class} if it can be written in the form
\begin{equation}
\label{eq:polya-szego}
F(z) = z + T - \int_{\mathbb{R}} \frac{d\mu(a)}{z - a},
\end{equation}
where $T \in \mathbb{R}$ and $\mu \ge 0$ is a finite singular measure on the real line. (To be honest, Polya and Szeg\"o considered the case when $\mu$ is supported on finitely many atoms, in which case $F$ is a rational function.) We refer to $T$ as the {\em Epstein phase} in honour of A.~Epstein \cite{epstein}, who used a similar normal form to study quadratic rational maps. 

\begin{remark}
For a holomorphic self-map $f$ of the unit disk with a Denjoy-Wolff fixed point $p \in \partial \mathbb{D}$, the Polya-Szeg\"o condition says that
$$
f(z) = z + (z-p) + c_2(z-p)^2 + c_3(z-p)^3 + o((z-p)^3), \qquad c_2, c_3 \in \mathbb{R},
$$
as $z \to p$ within a Stolz angle.
\end{remark}

From a dynamical perspective, general parabolic inner functions can be rather complicated, but Polya-Szeg\"o inner functions exhibit simpler behaviour, similar to that of finite Blaschke products.

From Julia's lemma, it follows that the {\em height} $\im z$ increases under forward iteration. This can also be seen directly from the explicit formula
\begin{equation}
\label{eq:imaginary-part-increasing}
\im F(z) - \im z \, = \, (1 + a^2) \int_{\mathbb{R}} \frac{\im z}{|z - a|^2} \, d\nu(a) \, > \, 0.
\end{equation}
We say that an inner function $F$ has {\em finite height} if for some (and hence every) forward orbit $(z_n)_{n=0}^\infty$, the imaginary parts $\im z_n$ are bounded. Otherwise, we say that $F$ has {\em infinite height}.

According to \cite[Lemma 3.1]{aaronson78} and \cite[Theorems 4.1 and 4.2]{doering-mane}, the dynamics of $F$ on the real line is governed by the convergence or divergence of the series
\begin{equation}
\label{eq:aaronson-sum-def}
\sum_{n=0}^\infty \frac{\im z_n}{|z_n|^2}, \qquad z_0 \in \mathbb{H}, \qquad z_n = F^{\circ n}(z_0).
\end{equation}
More precisely, if the series converges, then $F$ is recurrent on the real line, while if the series diverges, then under iteration, a.e.~point on the real line converges to the Denjoy-Wolff point. By the Schwarz lemma, the convergence of the series does not depend on the choice of $z_0 \in \mathbb{H}$.

\begin{theorem}
\label{adm2-thm}
For Polya-Szeg\"o inner functions, one has the following additional characterizations of simple and double parabolicity:

\begin{table}[h!]
\centering
\begin{tabular}{|c|p{6cm}|p{6cm}|}
\hline
 & \textbf{Simply parabolic} & \textbf{Doubly parabolic} \\ \hline
5 & Epstein phase $T \ne 0$ & Epstein phase $T = 0$ \\ \hline
6 & Finite height: $\lim_{n \to \infty} \im z_n < \infty$ & Infinite height: $\im z_n \to \infty$ \\ \hline
7 & $F^{\circ n}(x) \to \infty$ for a.e.~$x \in \mathbb{R}$ & $F: \mathbb{R} \to \mathbb{R}$ is recurrent \\ \hline
\end{tabular}
\caption{The dichotomy between simply and doubly parabolic inner functions (ctd.)}
\label{tab:adm2}
\end{table}
\end{theorem}

The proof of Theorem \ref{adm2-thm} is split between Lemma \ref{structure-orbits-T0} and Corollary \ref{divergence-corollary}, which address the case $T=0$, and Lemma \ref{pos-step-finite-height} and Corollary \ref{convergence-lemma}, which handle the case $T\neq 0$.

\begin{remark}
(i) A finite Blaschke product is doubly parabolic if and only if its Julia set is the extended real line $\mathbb{R} \cup \{ \infty \}$.

(ii) For general parabolic inner functions, one always has (Finite height) $\Rightarrow$ (Positive step) $\Rightarrow$ (Convergent).
For the first implication, suppose that $(z_n)_{n=0}^\infty$ is a forward orbit with $\lim_{n \to \infty} \im z_n = c$. Let $w_0 \in \mathbb{H}$ be a point in the upper half-plane with $\im w_0 > c$ and $w_n = F^{\circ n}(w_0)$ be its forward orbit. Since height is increasing under iteration, $\im w_n > \im w_0$ for all $n$. As the hyperbolic distance between $z_n$ and $w_n$ is bounded away from zero, $F$ cannot be zero step by \cite[Theorem G]{doering-mane}. The second implication is given in \cite[Corollary 4.3]{doering-mane}.
\end{remark}

We have the following dynamical characterization of simply parabolic Polya-Szeg\"o inner functions:

\begin{theorem}
\label{dyn-PS}
Let $F: \mathbb{H} \to \mathbb{H}$ be a parabolic inner function, with the parabolic fixed point at infinity. Then, $F$ has a bi-infinite orbit $(z_n)_{n=-\infty}^\infty$ with
$$
0 \, < \alpha \, := \, \lim_{n \to -\infty} \im z_{n} \, < \, \lim_{n \to \infty} \im z_{n} \, =: \, \beta \, < \, \infty
$$
and
$$
\lim_{n \to -\infty} d_{\mathbb{H}}(z_{n}, z_{n+1}) < \infty
$$
if and only if $F$ is a Polya-Szeg\"o inner function with $T \ne 0$.
\end{theorem}

The ``only if'' direction will be proved in Lemma \ref{dyn-PS-only-if}, while the ``if'' direction will be proved in Corollary \ref{PS-existence-of-bi-infinite-orbits}.

\subsection{Forward and backward cylinders}
\label{sec:fbc-intro}

We now focus on simply parabolic Polya-Szeg\"o inner functions. Without loss of generality, we may assume that $T > 0$. 
With help of the canonical conjugacies of Ch.~Pommerenke \cite{pommerenke} and P.~Poggi-Corradini \cite{pietro}, we can associate forward and backward quotient half-cylinders 
$$C_+, C_- \cong  \mathbb{H}/(z \to z + T)$$ to $F$. The following theorem describes the forward half-cylinder:

\begin{theorem}
\label{forward-cylinder-thm}
Let $F$ be a simply parabolic Polya-Szeg\"o parabolic inner function with $T > 0$. For any $z \in \mathbb{H}$, the limit
$$
\widetilde{P}_+(z) = \lim_{n \to \infty} F^{\circ n}(z) -  \re F^{\circ n}(i)
$$
exists and satisfies the functional equation $$\widetilde{P}_+(F(z)) = \widetilde{P}_+(z) + T.$$ The projection map $P_+: \mathbb{H} \to C_+$ given by
$
P_+(z) = \widetilde P_+(z) \!\! \mod T
$
is surjective. If $P_+(z_1) = P_+(z_2)$, then $F^{\circ k_1}(z_1) = F^{\circ k_2}(z_2)$ for some $k_1, k_2 \ge 0$.
\end{theorem}

Since inverse iteration not being globally defined, to construct the backward cylinder $C_-$, we work with inverse orbits. 
 We say that an inverse orbit ${\bf z} = (z_{-n})_{n=0}^\infty$ has {\em positive height} if $\im z_{-n}$ is bounded below by a positive constant. In the present setting, this condition implies that
 $$z_{-n} - z_{-n-1} \to T, \qquad \text{as }n \to \infty.$$
In particular, every such inverse orbit $(z_{-n})_{n=0}^\infty$ is a {\em backward iteration sequence with bounded steps} (BISBS):
 $$\lim_{n \to -\infty} d_{\mathbb{H}}(z_n, z_{n+1}) < \infty.$$
Conversely by \cite[Theorem 1.21]{pietro}, for any simply parabolic inner function, every BISBS  $(z_{-n})_{n=0}^\infty$ with $z_{-n} \to \infty$ as $n \to \infty$ has positive height.
 
\begin{theorem}
\label{backward-dynamics3}
Let $F$ be a simply parabolic Polya-Szeg\"o parabolic inner function with $T > 0$.
 Fix an inverse orbit ${\bf c} = (c_{-n})_{n=0}^\infty$ of positive height. For any inverse orbit of positive height ${\bf z} = (z_{-n})_{n=0}^\infty$, the limit
\begin{equation}
\label{eq:Pplus-lemma}
\widetilde{P}_-({\bf z}) = \lim_{n \to \infty} \bigl ( z_{-n} - \re c_{-n} \bigr )
\end{equation}
exists and satisfies the functional equation
$$
\widetilde{P}_-(\widehat{F}({\bf z})) = \widetilde{P}_-({\bf z}) + T,
$$
where $\widehat{F}$ is the map which applies $F$ to each coordinate. Furthermore, for any $w \in C_- \cong \mathbb{H}/(z \to z + T)$, there exists an inverse orbit ${\bf z}$ with
 $
w = P_-({\bf z}) := \widetilde{P}_-({\bf z}) \!\! \mod T.
$
This inverse orbit is unique up to shifting the coordinate, i.e.~replacing $z_n$ with $z_{n+k}$ for some $k \in \mathbb{Z}$.
\end{theorem}

\subsection{Lavaurs semigroups}
\label{sec:lavaurs-intro}

Consider the set
 $$
\mathscr S = \bigl \{ (n, k) : n \ge 0, k = 0 \text{ or } n \in \mathbb{Z}, k \ge 1 \bigr \},
$$
equipped with the order $(n_1, k_1) < (n_2, k_2)$ if $k_1 < k_2$ or $k_1 = k_2$, $n_1 < n_2$. For each $\sigma \in \mathbb{R}$, we construct an Abelian semigroup 
$$
\mathcal G \, = \, \mathcal G_\sigma \, = \, \{ G_{n,k} \, : \, (n, k) \in \mathscr S \} \, \subset \, \hol(\mathbb{H}, \mathbb{H}),
$$
 which contains the map $F = G_{1,0}$. Intuitively, the set
$\{ G_{n,k}(z) \}$ describes the {\em extended forward orbit} of $z$. 

\begin{figure}[h]
\centering
\includegraphics[scale=0.4]{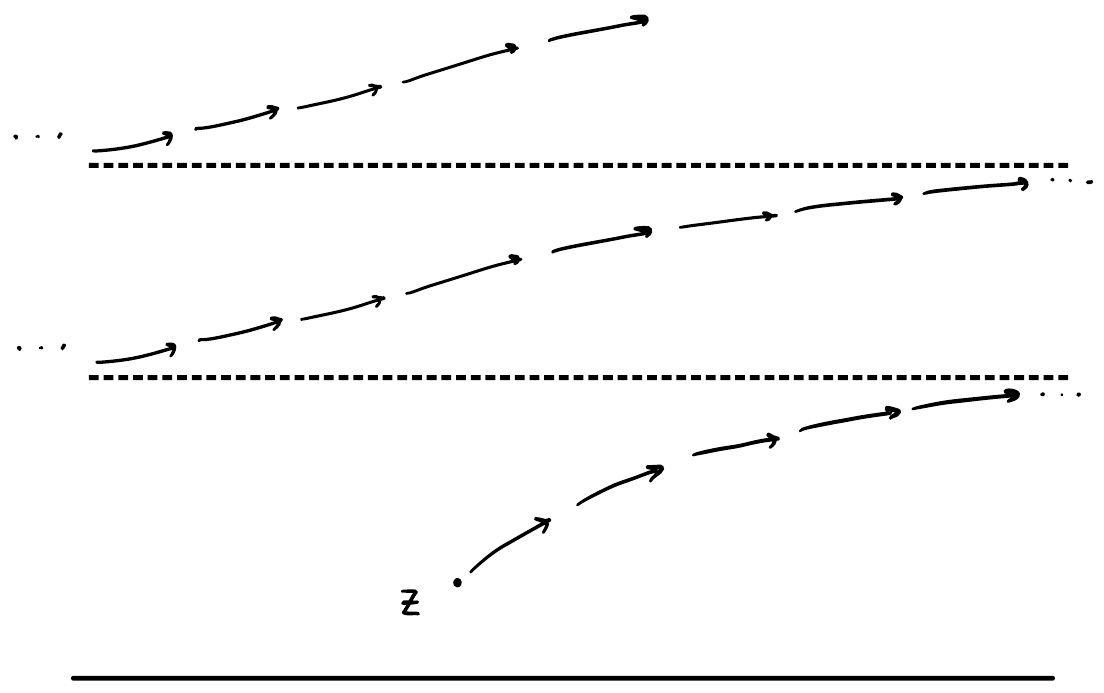}
 \caption{Iteration under a Lavaurs semigroup}
 \label{fig:curvature-lemma1}
\end{figure}

We now briefly explain how the Lavaurs semigroup comes about.
When we start iterating, the points $$z \to F(z) \to F^{\circ 2}(z) \to \dots$$
start moving right with asymptotic spacing $T$, while the imaginary parts of $\im F^{\circ n}(z)$ get ``stuck'' at a finite height. As discussed above,
$$
P_+(z) = \lim_{n \to \infty} F^{\circ n}(z) - \re F^{\circ n}(i) \!\! \mod T
$$
exists and defines a point in the forward half-cylinder $C_+ \cong \mathbb{H}/(z \to z + T)$. We can ``continue'' this orbit ``in the forward direction'' by a bi-infinite orbit 
${\bf z_1} = (z_{n,1})_{n=-\infty}^\infty$ which satisfies
$$
P_-({\bf z_1}) = P_+(z) + \sigma.
$$
However, in the forward direction the imaginary parts $\im z_{n, 1}$ remain bounded above, so $(z_{n, 1})$ can be continued by another bi-infinite orbit $(z_{n, 2})$ and so on.
The set of points obtained in the way constitute the extended orbit of $z$. For $n, k \in \mathscr S$, we set $G_{n,k}(z) := z_{n,k}$. For convenience write $G := G_{0,1}$.

\begin{remark}
Even though the labels of the maps $n, k$ depend on $\sigma \in \mathbb{R}$, the semigroup $\mathcal G$ only depends on $\sigma$ modulo $T$. We refer to the projection of $\sigma$ to 
$\mathbb{R}/T\mathbb{Z}$ as the {\em Lavaurs phase}\/.
\end{remark}

The map $G = G_\sigma$ induces a holomorphic self-map $L = L_\sigma$ of the forward cylinder $C_+$, which we call the {\em Lavaurs map}\/:
$$
L(P_+(z)) = P_+(G(z)), \qquad z \in \mathbb{H}.
$$
Since $C_+$ is conformally equivalent to a punctured disk $\mathbb{D}^* = \mathbb{D} \setminus \{ 0 \}$, we may view $L$ as a holomorphic self-map of the punctured disk. The following theorem summarizes the basic properties of Lavaurs semigroups and Lavaurs maps:

\begin{theorem}
\label{summary1}
Let $F$ be a simply parabolic Polya-Szeg\"o parabolic inner function with $T > 0$.

{\em (i)} The map $G$ is an infinite height parabolic self-map of the upper half-plane which commutes with $F$.

{\em (ii)} The Lavaurs map $L: \mathbb{D}^* \to \mathbb{D}^*$ has a removable singularity at the origin. The extension $L: \mathbb{D} \to \mathbb{D}$ is a holomorphic self-map of the unit disk with an attracting but not super-attracting Denjoy-Wolff fixed point at the origin. 

{\em (iii)} The quotient $\mathbb{H} / \mathcal G \cong \mathbb{D}^* / L$ is a torus.
\end{theorem}

 Let $(\widehat{F}, \widehat{\mathbb{R}}, \widehat{\ell})$ denote the natural extension of $(F, \mathbb{R}, \ell)$. A point in the natural extension $\widehat{\mathbb{R}}$ is a bi-infinite
 orbit ${\bf x} = (x_{n})_{n=-\infty}^\infty$. The map $\widehat{F}$ acts on $\widehat{\mathbb{R}}$ by applying $F$ to each coordinate. Finally, the natural extension measure $\widehat{\ell}$ is the unique 
 $\widehat{F}$ invariant measure on $\widehat{\mathbb{R}}$ that maps to $\ell$ under each coordinate.
  
\begin{theorem}
\label{summary2}
If additionally $\int_{\mathbb{R}} \log |F'| d\ell < \infty$, then:

{\em (i)} The maps $G: \mathbb{H} \to \mathbb{H}$ and $L: \mathbb{D} \to \mathbb{D}$ are inner functions.

{\em (ii)} The Lyapunov exponents of $F$ and $L$ are equal:
$$\int_{\mathbb{R}} \log |F'| d\ell = \int_{\partial \mathbb{D}} \log |L'| dm.$$

{\em (iii)} The Lebesgue measure on the real line $\ell$ is ergodic with respect to the Lavaurs semigroup $\mathcal G$.

{\em (iv)} For $\ell$ a.e.~$x \in \mathbb{R}$, $\lim_{n \to \infty} F^{\circ n}(x) - \re F^{\circ n}(i)$ exists. 

{\em (v)} For $\widehat{\ell}$ a.e.~inverse orbit ${\bf x} = (x_{-n})_{n=0}^\infty \in \widehat{\mathbb{R}}$, $\lim_{n \to \infty}  ( x_{-n}  - \re c_{-n} )$ exists, where
${\bf c}  = (c_{-n})_{n=0}^\infty$ is a fixed inverse orbit from Theorem \ref{backward-dynamics3}.
\end{theorem}

An example in Appendix \ref{sec:bad-example} shows that $G$ and $L$ can fail to be inner for arbitrary simply parabolic Polya-Szeg\"o inner functions.

\begin{remark}
Historically, Lavaurs maps were introduced to study of perturbations of parabolic rational maps such as $z \to z^2+1/4$, see \cite{lavaurs, douady, oudkerk}. One intriguing difference between our work and the classical parabolic implosion literature is that parabolic basins of rational maps are conjugate to doubly parabolic inner functions, while we are primarily concerned with simply parabolic inner functions.
This leads to a different behaviour of commuting holomorphic maps: for $F(z) = z^2 + 1/4$, the maps $G_\sigma$ attempt to leave the cauliflower (by squeezing through the gate at the parabolic fixed point 1/2), while in our setting, the  maps $G_\sigma$ preserve the upper and lower half-planes.
\end{remark}

\subsection{Riemann surface laminations}

Suppose $f: \mathbb{D} \to \mathbb{D}$ is a hyperbolic inner function with Denjoy-Wolff point $p \in \mathbb{D}$. Let
$$
\widehat{\mathbb{D}} \, = \, \bigl \{ (z_n)_{n=-\infty}^\infty: z_n \in \mathbb{D}, \ z_{n+1} = f(z_{n}), \  |z_n| \to 1 \text{ as }n \to -\infty \bigr \}
$$
be the space of bi-infinite orbits of $f$ with the constant orbit of the Denjoy-Wolff point removed and $\widehat{f}: \widehat{\mathbb{D}} \to \widehat{\mathbb{D}}$ be the map which applies $f$ to each coordinate. The Riemann surface lamination associated to $f$ is defined as the quotient 
$
\widehat{X}_f = \widehat{\mathbb{D}} \setminus \widehat f.
$ 
According to Sullivan's dictionary, $\widehat{X}_f$ is an analogue of the unit tangent bundle $T_1 \Sigma$ of a compact Riemann surface $\Sigma$.

In the case when $f$ is a finite Blaschke product, McMullen \cite{mcmullen} defined a  volume form $\xi$ on $\widehat{X}_f$ as well as a geodesic flow $g_t: \widehat{X}_f \to \widehat{X}_f$, and proved that this flow is ergodic. 

Together with M.~Urba\'nski \cite{laminations}, the author showed that the geodesic flow $g_t: \widehat{X}_f \to \widehat{X}_f$ is ergodic for any hyperbolic inner function with 
$
\int_{\partial \mathbb{D}} \log |f'| d\omega_p < \infty,
$
where $\omega_p$ is the harmonic measure on the unit circle viewed from $p$. This assumption is natural since it guarantees that the total mass of $\xi$ is finite -- in fact,
$$\xi(\widehat{X}_f) = \int_{\partial \mathbb{D}} \log |f'| d\omega_p.$$ In the same paper, similar results were also obtained for doubly parabolic inner functions $F: \mathbb{H} \to \mathbb{H}$ satisfying $\int_{\mathbb{R}} \log |F'| d\ell < \infty$.

%McMullen used the ergodicity of the geodesic flow to study an analogue of the Weil-Petersson metric on spaces of Blaschke products, while the main application

It therefore remains to investigate the case of simply parabolic inner functions.
It is not difficult to show, and we will do so in Corollary \ref{Finite-Lyapunov-Polya-Szego}, that if a parabolic inner function $F$ has a finite Lyapunov exponent $\int_{\mathbb{R}} \log |F'| d\ell$, then it belongs to the Polya-Szeg\"o class, so the notions defined in Sections \ref{sec:fbc-intro} and \ref{sec:lavaurs-intro} make sense. Let $\mathcal G$ be the Lavaurs semigroup associated to $F$ and $\sigma \in \mathbb{R}/T\mathbb{Z}$. 
Recall from Section \ref{sec:lavaurs-intro} that the map $G_\sigma$ induces a hyperbolic inner function $L = L_\sigma: \mathbb{D} \to \mathbb{D}$ with a Denjoy-Wolff fixed point at  the origin.

In Section \ref{sec:laminations}, we associate a Riemann surface lamination $\widehat{X}_{\mathcal G}$ to $\mathcal G$.
The definition is similar to the ones for the hyperbolic and doubly parabolic inner functions but involves extended inverse orbits.
We notice that $\widehat{X}_{\mathcal G}$ can be naturally identified with $\widehat{X}_L$. With help of this identification, we transfer the volume form and geodesic flow from $\widehat{X}_L$ to $\widehat{X}_{\mathcal G}$. 
By Theorem \ref{summary2}(ii), the lamination $\widehat{X}_{\mathcal G}$ has total mass
 $$\xi_{\mathcal G}(\widehat{X}_{\mathcal G}) = \int_{\mathbb{R}} \log |F'| d\ell.$$
Since the geodesic flow on ${\widehat X}_L$ is ergodic, the geodesic flow on $\widehat{X}_{\mathcal G}$ is also ergodic.
Following the strategy laid out in \cite{laminations}, we use the ergodicity of the geodesic flow to study orbit counting:

\begin{theorem}
\label{main-thm3c}
For a bounded interval $I$ in the real line, let
$$
\mathcal N_I(z, R) = \# \bigl \{ w \in I \times [e^{-R}, 1] : G_{n,k}(w) = z \text{ for some }G_{n,k} \in \mathcal G \bigr \}.
$$
For $z \in \mathbb{H}$ outside a set of zero measure, we have
$$
\frac{1}{R} \int_0^R \frac{\mathcal N_I(z, S)}{e^S} dS \, \sim \,  |I| \cdot \frac{1}{\int_{\mathbb{R}} \log |F'| d\ell}
$$
as $R \to \infty$.
\end{theorem}

\subsection{Parabolic one component inner functions}

An inner function $f: \mathbb{D} \to \mathbb{D}$ is called a {\em one component inner function}  if the set $$\{ z \in \mathbb{D} : |f(z)| < \rho \}$$ is connected for some $0 < \rho < 1$. A point $v \in \mathbb{D}$ is a {\em regular value} of $f$ if $f$ is a covering map over some open neighbourhood of $v$.
Points of the unit disk that are not regular values are called {\em singular values}\/.  Equivalently, an inner function 
is a one component inner function if the set of singular values is compactly contained in the unit disk. For further characterizations of one component inner functions and the equivalence of the two definitions, see \cite[Section 7]{inner-tdf}.

When dealing with parabolic inner functions, it is convenient to use a slightly different notion. A parabolic inner function $F: \mathbb{H} \to \mathbb{H}$ if called a {\em parabolic one component inner function} 
if the set $$F^{-1}(\mathcal H_{\rho}) = \{ z \in \mathbb{H} : \im F(z) > \rho \}$$ is connected for some $\rho > 0$. Equivalently, a parabolic inner function is a parabolic one component inner function if the set of singular values is contained in a horoball $\mathcal H_{\rho}$ for some $\rho > 0$. The equivalence of the two definitions is proved in the same way as for classical one component inner functions. 

\begin{lemma}
\label{p1c-structure}
Suppose $F: \mathbb{H} \to \mathbb{H}$ is a parabolic one component inner function.

{\em (i)} If $F$ has a finite Lyapunov exponent, $\int_{\mathbb{R}} \log |F'| d\ell < \infty$, then the measure $\mu$ in the Polya-Szeg\"o representation is compactly supported. Consequently, after conjugating to the unit disk, the corresponding map $f: \mathbb{D} \to \mathbb{D}$ extends analytically in a neighbourhood of the Denjoy-Wolff point $p$. 

{\em (ii)} If in addition, $F$ is simply parabolic, then for any $\sigma \in \mathbb{R}/T\mathbb{Z}$, the Lavaurs map $L_\sigma$ is a classical one component inner function.
\end{lemma}

As explained in \cite{laminations}, for one component inner functions with a Denjoy-Wolff fixed point in the unit disk not conjugate to $z^d$ for some $d \ge 2$, the geodesic flow on the Riemann surface lamination ${\widehat X}_{L_\sigma}$ is mixing. Since $\widehat{X}_{\mathcal G} \cong {\widehat X}_{L_\sigma}$, the geodesic flow on ${\widehat X}_{\mathcal G}$ is also mixing. The arguments in \cite{laminations} now yield the stronger conclusion:

\begin{theorem}
\label{main-thm3d}
For $z \in \mathbb{H}$ outside a set of zero measure, we have
$$
\mathcal N_I(z, R) \, \sim \,  |I| \cdot \frac{e^R}{\int_{\mathbb{R}} \log |F'| d\ell}
$$
as $R \to \infty$.
\end{theorem}

\section{Canonical conjugacies}
\label{sec:canonical-conjugacies}

In this section, we review the canonical conjugacies introduced by Ch.~Pommerenke \cite{pommerenke} and P.~Poggi-Corradini \cite{pietro}. %For simplicity, we state the results for parabolic inner functions; however, they hold for all holomorphic self-maps of the upper half-plane with a parabolic fixed point at infinity.

\begin{theorem}[Pommerenke]
\label{pommerenke-conjuacy}
Suppose $F: \mathbb{H} \to \mathbb{H}$ is a parabolic self-map of the upper half-plane and $(z_n)_{n=0}^\infty$ is a forward orbit with positive step: $$\lim_{n \to \infty} d_{\mathbb{H}}(z_n, z_{n+1}) > 0.$$
Write: $F^{\circ n}(i) = x_n + iy_n$. 
Then the rescaled iterates
$$
\phi(z) = \lim_{n \to \infty} \frac{ F^{\circ n}(z) - \re F^{\circ n}(i) }{y_n}
$$
converge to a holomorphic function $\phi: \mathbb{H} \to \mathbb{C}$ which satisfies the functional equation
$$
\phi(F(z)) = \phi(z) + b, \qquad
\text{where }
b = \lim_{n \to \infty} \frac{x_{n+1} - x_n}{y_n}.
$$
\end{theorem}

If the imaginary parts $y_n = \im z_n$ are bounded above, then dividing by $y_n$ is not necessary and Pommerenke's theorem takes the following simplified form:

\begin{corollary}
\label{pommerenke-conjuacy2}
Suppose $F: \mathbb{H} \to \mathbb{H}$ is a parabolic self-map of the upper half-plane and $(z_n)_{n=0}^\infty$ is a forward orbit of finite height, i.e.~$\lim_{n \to \infty} \im z_n < \infty$.  Then $\lim_{n \to \infty} d_{\mathbb{H}}(z_n, z_{n+1}) > 0$ and
$$
\widetilde{P}_+(z) = \lim_{n \to \infty}  \bigl ( F^{\circ n}(z) - \re F^{\circ n}(i) \bigr )
$$
satisfies the functional equation
$$
\widetilde{P}_+(F(z)) = \widetilde{P}_+(z) + b, \qquad
\text{where }
b = \lim_{n \to \infty} (x_{n+1} - x_n).
$$
\end{corollary}

We now turn to backward iteration.

\begin{theorem}[Poggi-Corradini]
\label{pietro-conjuacy}
Suppose $F: \mathbb{H} \to \mathbb{H}$ is a parabolic self-map of the upper half-plane and $(z_{-n})_{n=0}^\infty$ is a {\em backward iteration sequence with bounded steps} (BISBS):
 $$\lim_{n \to -\infty} d_{\mathbb{H}}(z_n, z_{n+1}) < \infty.$$
  Write: $z_{-n}= x_{-n} + iy_{-n}$.
Then the rescaled iterates
 $$
\psi(z) = \lim_{n \to \infty} F^{\circ n}(x_{-n} + y_{-n} z)
$$
converge to a holomorphic function $\psi: \mathbb{H} \to \mathbb{C}$ which satisfies the functional equation
$$
\psi(z + b) = F(\psi(z)), \qquad
\text{where }
b = \lim_{n \to -\infty} \frac{x_{n+1} - x_{n}}{y_n}.
$$
\end{theorem}

Similarly, if the imaginary parts $y_n = \im z_n$ are bounded below by a positive constant, then dividing by $y_n$ is not necessary:

\begin{corollary}
\label{pietro-conjuacy2}
Suppose $F: \mathbb{H} \to \mathbb{H}$ is a parabolic self-map of the upper half-plane and $(z_{-n})_{n=0}^\infty$ is a BISBS of non-zero height, i.e.~$\lim_{n \to -\infty} \im z_n > 0$. Then
$
\Psi(z) = \lim_{n \to \infty}  F^{\circ n}(x_{-n} + z)
$
satisfies the functional equation
$$
\Psi(z + b) = F(\Psi(z)), \qquad
\text{where }
b = \lim_{n \to -\infty} (x_{n+1} - x_{n}).
$$
\end{corollary}

\begin{remark}
(i) The theorems of Pommerenke and Poggi-Corradini do not imply that the limits $\lim_{n \to \infty} F^{\circ n}(i) - nT$
and $\lim_{n \to \infty} c_{-n} + nT$ exist, or that these quantities even remain bounded as $n \to \infty$.

(ii) If $\widetilde{P}_-$ is the inverse of $\Psi$ on some subdomain of $\mathbb{H}$, then $\widetilde{P}_-$ satisfies the same functional equation as $\widetilde{P}_+$\,:
$$
\widetilde{P}_-(F(z)) = \widetilde{P}_-(z) + b.
$$

(iii)
In \cite{hamilton-iii},  D.~Hamilton observed that if $F$ is a simply parabolic inner function, then the Pommerenke conjugacy $\widetilde{P}_+: \mathbb{H} \to \mathbb{H}$ is also an inner function. This result may also be deduced from the work of G.~Ferreira \cite{ferreira} on the limits of non-autonomous forward iterates of inner functions. From a recent work by D.~Kraus, A.~Moucha and O.~Roth \cite{KMR}, it follows that the Pommerenke conjugacy is an indestructible or maximal Blaschke product if and only if $F$ is. In contrast, backward iteration appears to be much more delicate and we were only able to show that $\Psi$ is an inner function by imposing much stronger restrictions on $F$.

\end{remark}

\section{The Polya-Szeg\"o class}

In this section, we show that if a parabolic inner function has finite Lyapunov exponent, then it belongs to the Polya-Szeg\"o class. We then show that for Polya-Szeg\"o inner functions with  $T = 0$, forward orbits have infinite height and the dynamics on the real line is recurrent. Finally, we show that if $F$ possesses a bi-infinite orbit satisfying the conditions of Theorem \ref{dyn-PS}, then $F$ is a simply parabolic Polya-Szeg\"o inner function.

\subsection{Basic observations}

The following lemma describes Polya-Szeg\"o functions as a subset of parabolic inner functions:

\begin{lemma}
\label{polya-szego-functions}
Let $F$ be a parabolic inner function of the form (\ref{eq:general-inner}). Then, $F$ is a Polya-Szeg\"o function if and only if
$$
(1+a^2) \nu(a)
$$
is a finite measure. In this case,
$$
\mu = (1+a^2) \nu(a), \qquad \text{and} \qquad
T = S - \int_{\mathbb{R}} a \, d\nu(a).
$$
\end{lemma}

\begin{corollary}
\label{Finite-Lyapunov-Polya-Szego}
Any parabolic inner function $F: \mathbb{H} \to \mathbb{H}$ with finite Lyapunov exponent
$
\int_{\mathbb{R}} \log |F'(x)| d\ell < \infty
$
belongs to the Polya-Szeg\"o class.
\end{corollary}

\begin{proof}
Differentiating (\ref{eq:general-inner}), we get
\begin{equation}
\label{eq:derivative-general}
F'(z) = 1 + \int_{\mathbb{R}} \frac{1+a^2}{(z - a)^2} \, d\nu(a).
\end{equation}
From the above formula, it is clear that
$$
\log F'(z) \gtrsim \min \Bigl (1, \, (1+n^2)\nu([n, n+1]) \Bigr ), \qquad z \in [n, n+1).
$$
The finiteness of the Lyapunov exponent forces the measure $(1+a^2)\nu(a)$ to be finite. The corollary now follows from Lemma \ref{polya-szego-functions}.
\end{proof}

\subsection{Polya-Szeg\"o inner functions with  $T = 0$}

We now take a brief look at Polya-Szeg\"o inner functions with $T = 0$.

\begin{lemma}
\label{structure-orbits-T0}
Let $F$ be a Polya-Szeg\"o inner function with $T = 0$. Every bi-infinite orbit $(z_n)_{n=-\infty}^\infty$ satisfies 
$$
\im_{n \to -\infty} \im z_{n} = 0, \qquad \lim_{n \to \infty} \im z_{n} = \infty,
$$
$$
 \lim_{n \to \infty}  d_{\mathbb{H}}(z_n, z_{n+1}) = 0.
$$
\end{lemma}

\begin{proof}
The inequality
$$
|F(z) - z| \le \int_{\mathbb{R}} \frac{d\mu(a)}{|z-a|}
$$
shows that the hyperbolic distance $d_{\mathbb{H}}(z, F(z)) \to 0$ as $|z| \to \infty$, provided that $\im z$ stays bounded below by a positive constant.

Since any forward orbit $(z_n)_{n=0}^\infty$ converges to the Denjoy-Wolff point at infinity and $\im z_n \ge \im z_0$ is bounded below, 
this implies that the hyperbolic distance $$d_{\mathbb{H}}(z_n, z_{n+1}) \to 0, \qquad \text{as }n \to \infty.$$
As discussed in the introduction, the zero step condition implies that forward orbits of $F$ have infinite height.

Suppose that $(z_n)_{n=-\infty}^0$ is a backward orbit such that  $\im z_{n}$ are bounded below by a positive constant. Since $(z_n)_{n=-\infty}^0$ cannot be trapped in a compact subset of $\mathbb{H}$, there exists a subsequence $n_k \to \infty$ such that $$d_{\mathbb{H}}(z_{n_k}, z_{n_k+1}) \to 0, \qquad \text{as }n \to -\infty,$$ which contradicts the Schwarz lemma.
\end{proof}

The following lemma appears in \cite[Proposition 4.1]{AS21}:

\begin{lemma}
\label{structure-orbits-T0-up}
Suppose $F$ is a Polya-Szeg\"o inner function with $T = 0$. Then, every forward orbit $(z_n)_{n=0}^\infty$ tends to infinity non-tangentially, 
that is, $$\liminf_{n \to \infty} \frac{\im z_n}{\re z_n} > 0.$$
\end{lemma}

We note that the case when $\mu$ is compactly supported has been considered in \cite[Theorem 6.4.1]{aaronson-book}.

\begin{corollary}
\label{divergence-corollary}
Suppose $F$ is a Polya-Szeg\"o inner function with $T = 0$. Then, $F$ is recurrent on the real line.
\end{corollary}

\begin{proof}
Lemma \ref{structure-orbits-T0-up} implies that $\im z_n \asymp n^{1/2}$ and $\frac{\im z_n}{|z_n|^2} \asymp n^{-1/2}$. As the Aaronson sum (\ref{eq:aaronson-sum-def}) is finite, $F$ is recurrent on the real line.
\end{proof}

\subsection{Dynamical characterization of simply parabolic Polya-Szeg\"o inner functions}
\label{sec:dyn-char-T0}

We now show the ``only if'' direction of Theorem \ref{dyn-PS}:

\begin{lemma}
\label{dyn-PS-only-if}
Suppose $F: \mathbb{H} \to \mathbb{H}$ is parabolic inner function. If $(z_n)_{n=-\infty}^\infty$ is a BISBS with $\im z_n \in (\alpha, \beta)$, $\alpha, \beta > 0$, then $F$ is a Polya-Szeg\"o inner function with $T \ne 0$.
\end{lemma}

\begin{proof}
{\em Step 1.}  We write $z_n = x_n + iy_n$. In this step, we show that
as $n \to \infty$, either $$x_n \to \infty, \quad x_{-n} \to -\infty \qquad \text{or} \qquad x_n \to -\infty, \quad x_{-n} \to \infty.$$
Let $J \subset \mathbb{R}$ be a bounded interval in the real line. By compactness, for any point $z$ in $J \times [\alpha, \beta]$, the imaginary part of $z$ increases by a definite amount:
$$y_{n+1} - y_n > c(J, F) > 0.$$
Consequently, only finitely many points of bi-infinite orbit $(z_n)_{n=-\infty}^\infty$ can lie inside $J  \times [\alpha, \beta]$.

If $J$ is sufficiently long so that $\diam_{\hyp} (J \times [\alpha, \beta]) > \sup_{n \in \mathbb{Z}} d_{\mathbb{H}}(z_n, z_{n+1})$, then the real parts $x_n = \re z_n$ are eventually contained in one of the two components of $\mathbb{R} \setminus J$.
Therefore, we either have $x_n \to \infty$ or $x_n \to -\infty$.
The same argument also shows that $x_{-n} \to -\infty$ or $x_{-n} \to \infty$. 

Without loss of generality, we assume that $x_n \to \infty$.
If additionally $x_{-n} \to \infty$ as $n \to \infty$, then there would exist arbitrarily large integers $m,n$ for which the points $z_n$ and $z_{-m}$ are a bounded hyperbolic distance apart. However, by the Schwarz lemma, this would imply that the hyperbolic distance between $x_0$ and $x_{n+m}$ remains bounded, which is a contradiction.

\medskip

{\em Step 2.}
From Step 1, it follows that there exists a $\tau > 0$ so that each interval $I \subset \mathbb{R}$ of length $\tau$ contains $\re z_n$ for some $n \in \mathbb{Z}$.
According to (\ref{eq:imaginary-part-increasing}), the increment of the imaginary part in a single step is
$$
\im (z_{n+1} - z_n) = \int_{\mathbb{R}} (1 + a^2) \frac{\im z}{|z - a|^2} \, d\nu(a).
$$
Since $\alpha \le \im z_n \le \beta$,
$$
\im (z_{n+1} - z_n) \gtrsim (1+a^2) \nu(I).
$$
However, this forces $(1+a^2) \nu(I)$ to be a finite measure: otherwise, the total increment 
$$
\sum_{n=-\infty}^\infty \im (z_{n+1} - z_n)
$$
would be infinite. By Lemma \ref{polya-szego-functions}, $F$ is a Polya-Szeg\"o inner function, while Lemma \ref{structure-orbits-T0} implies that $T \ne 0$.
\end{proof}

\section{\texorpdfstring{Polya-Szeg\"o inner functions with $T \ne 0$}{Polya-Szeg\"o inner functions with T ≠ 0}}                         
\label{sec:1p-polya-szego}

In this section, we examine Polya-Szeg\"o inner functions with $T \ne 0$. Without loss of generality, we assume that $T > 0$ as the case $T < 0$ is similar. Throughout this section, we treat 
$\mu(\mathbb{R})$ and $T$ as fixed constants, and therefore, we do not indicate the dependence of the error terms on $\mu(\mathbb{R})$ and $T$. We first examine the behaviour of forward and backward orbits when $\im z$ is large, before looking at the general case. 

With help of our analysis, we complete the proof of Theorem \ref{adm2-thm} by showing that forward orbits have finite height and under iteration, a.e.~point on the real line converges to infinity. We also complete the proof of Theorem \ref{dyn-PS} by constructing a BISBS $(z_n)_{n=-\infty}^\infty$ with $\im z_n \in (\alpha, \beta)$, $\alpha, \beta > 0$. Finally, we also prove Theorems \ref{forward-cylinder-thm} and \ref{backward-dynamics3} on forward and backward half-cylinders. These theorems mostly follow from the canonical conjugacies of Pommerenke and Poggi-Corradini described in Section \ref{sec:canonical-conjugacies}, except for the issue of surjectivity.

\subsection{Basic estimates}

The estimates
$$
\label{eq:explicit-formula}
F(z) - z - T \, \le \, \int_{\mathbb{R}} \frac{d\mu(a)}{|z - a|}
\qquad
\text{and}
\qquad 
F'(z) - 1 \, \le \, \int_{\mathbb{R}} \frac{d\mu(a)}{|z - a|^2}
$$
imply that $F(z)$ is close to a translation by $T$ on a large subset of the upper half-plane:

\begin{lemma}
\label{basic-estimate-new}
For every $\delta > 0$, there exists an $R > 0$ such that 
\begin{equation}
\label{eq:forward-iteration-is-nice-new}
   |F(z) - z - T| < \delta
 \qquad \text{and} \qquad
 |F'(z) - 1| < \delta,
 \end{equation}
whenever $\im z > \delta$ and $|z| > R$.
\end{lemma}
 
If $\delta > 0$ is chosen sufficiently small, for instance $0 < \delta < 2/\pi$, then the second estimate
in (\ref{eq:forward-iteration-is-nice-new}) implies that
$F$ is injective on
$$
\Delta_{\delta, R} \, = \, \bigl \{z \in \mathbb{H} \, : \, \im z > \delta, \, |z| > R \bigr \}.
$$
Indeed, given two points $z_1, z_2 \in \Delta_{\delta, R}$, one can connect them by a curve $\gamma \subset \Delta_{\delta, R}$ of length at most $(\pi/2)|z_2-z_1|$ and then apply the fundamental theorem of calculus to show that they have distinct images under $F$. We leave the details to the reader.
Meanwhile, the first estimate in (\ref{eq:forward-iteration-is-nice-new}) implies that $F(\Delta_{\delta, R}) \supset \Delta_{2\delta, R'}$ where $R' = R + T + \delta$. Therefore, one can define a branch of $F^{-1}$ on $\Delta_{2\delta, R'}$
which satisfies
\begin{equation}
\label{eq:natural-inverse}
 |F^{-1}(z) - (z-T)| < \delta, \qquad z \in \Delta_{2\delta, R'}.
\end{equation}

\begin{remark}
We say that $w \in \mathbb{H}$ is the {\em natural inverse} of $z \in \mathbb{H}$ if
\begin{equation}
\label{eq:natural-inverse2}
\im w > (1/2)\im z.
\end{equation}
As explained in Corollary \ref{exceptional3}, for any $z \in \mathbb{H}$, one has
$$
\sum_{F(w)=z} \im w \le \im z.
$$
Hence, there can be at most one pre-image of $z$ which satisfies (\ref{eq:natural-inverse2}). The inverse branch constructed in (\ref{eq:natural-inverse}) selects precisely this pre-image.
\end{remark}

\subsection{Imaginary part along forward orbits}
\label{sec:imaginary-forward}

From the definition of a Polya-Szeg\"o function (\ref{eq:polya-szego}), for any $z = x + iy \in \mathbb{H}$, we have
\begin{equation}
\label{eq:imaginary-increment-forward}
\im F(z) - \im z \, =\, - \im \int_{\mathbb{R}} \frac{d\mu(a)}{z - a} \, =\,  \int_{\mathbb{R}} \frac{y \, d\mu(a)}{(x-a)^2 + y^2}
\end{equation}
and
\begin{equation}
\label{eq:real-increment-forward}
\re F(z) - \re z - T  \, =\, - \re \int_{\mathbb{R}} \frac{d\mu(a)}{z - a} \, \le \,  \int_{\mathbb{R}} \frac{d\mu(a)}{|z - a|}.
\end{equation}
In particular, if $$\im z \, > \,  Y \, := \, (2/T) \cdot \mu(\mathbb{R})$$ is large, then the real part increases by an amount comparable to $T$, while the imaginary part increases by at most
$$
\int_{\mathbb{R}} \frac{Y \, d\mu(a)}{(x-a)^2 + Y^2}.
$$

Consider a forward orbit $(z_n)_{n = 0}^\infty$ with $\im z_0 > Y$. If $z_n = x_n + iy_n$, then
\begin{align*}
\sum_{n=0}^\infty (y_{n+1} - y_n) & \le \sum_{n= 0}^\infty \int_{\mathbb{R}} \frac{Y d\mu(a)}{(x_n-a)^2 + Y^2} \\
& \lesssim \int_{\mathbb{R}} \int_{\mathbb{R}} \frac{Y d\mu(a)}{(x-a)^2+Y^2} \, dx \\
& \lesssim \int_{\mathbb{R}}  \int_{\mathbb{R}} \frac{Y d\mu(a)}{x^2+Y^2} \, dx \\
& \lesssim \int_{\mathbb{R}}  \frac{Y}{x^2+Y^2} \, dx \\
& \le C,
\end{align*}
where $C > 0$ depends on $T$ and $\mu(\mathbb{R})$. Therefore, the imaginary parts $y_n$ remain bounded under iteration, while by Lemma \ref{basic-estimate-new}, the horizontal increments satisfy $x_{n+1} - x_n \to T$ as $n \to \infty$.
Intuitively, the forward orbit $(z_n)_{n = 0}^\infty$ resembles equally spaced points on a horizontal line at height  $y_\infty = \lim y_n$.

\begin{corollary}
\label{pos-step-finite-height}
A Polya-Szeg\"o inner function $F$ with $T \ne 0$ has positive step and finite height. 
\end{corollary}

\subsection{Imaginary increments along bi-infinite orbits}
\label{sec:imaginary-increments-bi-infinite-orbits}

Let $Y, C > 0$ be the constants from Section \ref{sec:imaginary-forward}.
On the half-plane $\{ \im z > Y + C \}$, the natural inverse $F^{-1}$ satisfies
$$
|\re F^{-1}(z) - \re z - T| < T/2
$$
and
$$
 |\im F^{-1}(z) - \im z| < \int_{\mathbb{R}} \frac{Y d\mu(a)}{(\re F^{-1}(z) -a)^2 + Y^2}.
$$
Let $z_0 \in \mathbb{H}$ with $\im z_0 > Y + 2C$. We define the first backward iterate $z_{-1} = F^{-1}(z_0)$ by using the natural branch of the inverse.
Since $\im z_{-1} > Y + C$, we can also define $z_{-2} = F^{-1}(z_1)$. Continuing in this way, we obtain an inverse orbit $(z_{-n})_{n=0}^\infty$ with $\im z_{-n} > Y + C$ bounded below.

\begin{corollary}
\label{PS-existence-of-bi-infinite-orbits}
Any point $z_0 \in \mathbb{H}$ with $\im z_0$ sufficiently large can be uniquely continued to a bi-infinite orbit $(z_n)_{n=-\infty}^\infty$ such that
\begin{equation}
\label{eq:nice}
|F(z_n) - z_n - T| < T/2, \qquad n \in \mathbb{Z},
\end{equation}
and
$$
\lim_{n \to \infty} \im z_n - \lim_{n \to -\infty} \im z_n < C.
$$
\end{corollary}

\begin{lemma}
\label{large-imaginary-part}
Let $F$ be a Polya-Szeg\"o inner function with $T \ne 0$. Suppose $(z_n)_{n=-\infty}^\infty$ is a bi-infinite orbit constructed in Corollary \ref{PS-existence-of-bi-infinite-orbits}.
As $\im z_0 \to \infty$, the imaginary increment
$$
\lim_{n \to \infty} \im z_n - \lim_{n \to -\infty} \im z_n \, \to \,  \frac{\pi}{T} \cdot \mu(\mathbb{R}).
$$
\end{lemma}

\begin{proof}
For any $\delta > 0$, we can request $y_0 = \im z_0$ to be sufficiently large, so that the horizontal spacing between the consecutive points $x_{n+1} - x_n \in (T - \delta, T + \delta)$. In the $n$-th step, the imaginary part increases by approximately
$$
y_{n+1} - y_n \, = \, \int_{\mathbb{R}} \frac{y_n \, d\mu(a)}{(x_n-a)^2 + y_n^2} \, \approx \, \int_{\mathbb{R}} \frac{y_0 \, d\mu(a)}{(x_n-a)^2 + y_0^2}.
$$
Summing these increments over $n \in \mathbb{Z}$ and replacing the discrete sum by an integral in $x$, we obtain
$$
\frac{1}{T} \int_{\mathbb{R}} \int_{\mathbb{R}} \frac{y_0 \, d\mu(a)}{(x-a)^2 + y_0^2} \, dx = \pi \cdot \frac{\mu(\mathbb{R})}{T},
$$
as desired.
\end{proof}

\subsection{Behaviour on the real line}

\begin{lemma}
\label{convergence-lemma}
Suppose $F$ is a Polya-Szeg\"o inner function with $T \ne 0$. Then, for a.e.~$x \in \mathbb{R}$, the iterates $F^{\circ n}(x) \to \infty$.
\end{lemma}

\begin{proof}
Suppose $(z_n)_{n=-\infty}^\infty$ is a bi-infinite orbit satisfying (\ref{eq:nice}). For large $n$, the points $z_n = x_n + iy_n$ are approximately equally spaced on the horizontal line $$\bigl \{ z \in \mathbb{H} : \im z = y_\infty = \lim_{n \to \infty} y_n \bigr \},$$ with spacing $T$. In particular, $\frac{\im z_n}{|z_n|^2} \asymp 1/n^2$ for $n \ge 1$. Since the Aaronson sum
(\ref{eq:aaronson-sum-def}) is finite, under iteration, a.e.~point on the real line converges to the Denjoy-Wolff point at infinity.
\end{proof}

\subsection{Forward iteration}

By Lemma \ref{basic-estimate-new}, far to the right, the map $F$ looks like translation by $T$. We now examine the long-term behaviour of the forward iterates:

\begin{lemma}
\label{forward-dynamics}
For every $0 < \delta < T/3$, there exists an $X > 0$ such that the region
$$
\Omega_{+, \delta} \, = \, \{ x + iy \, : \, x > X, \, y > \delta \}
$$
is forward-invariant and
\begin{equation}
\label{eq:forward-dynamics1}
|F(z) - z - T| < \delta, \qquad z \in \Omega_{+, \delta}.
\end{equation}
Moreover,
\begin{equation}
\label{eq:forward-dynamics2}
0 \, < \, \im F^{\circ n}(z) - \im z \, < \, \delta, \qquad n \ge 1, \qquad z \in \Omega_{+, \delta}.
\end{equation}
\end{lemma}

\begin{proof}
When $X > 0$ is large, the difference $F(z) - z - T$ is small by Lemma \ref{basic-estimate-new}, while the total increment of the imaginary part along the forward orbit can be estimated as in Section \ref{sec:imaginary-forward}. Since the real and imaginary parts are increasing on $\Omega_{+, \delta}$, the region $\Omega_{+, \delta}$ is forward-invariant.
\end{proof}

The following lemma says that every forward orbit eventually enters one of these forward-invariant regions:

\begin{lemma}
\label{forward-dynamics2}
For any $w\in \mathbb{H}$ with $\im w > \delta$, the forward orbit $w_n = F^{\circ n}(w)$ is eventually contained in the region $\Omega_{+, \delta}$.
\end{lemma}

\begin{proof}
Let $w \in \mathbb{H}$. By Corollary \ref{pos-step-finite-height}, the imaginary parts $\im w_n$ are bounded above by some constant $Y_\infty$.
By compactness, for any bounded interval $J \subset \mathbb{R}$, there exists a constant $c = c_J > 0$ such that
$$
\im F(z) - \im z > c, \qquad z \in J \times [\im w_0, Y_\infty].
$$
Consequently, the forward orbit of $w$ must eventually leave every compact set in the horizontal direction, i.e.~
$$
|\re w_n| \to \infty, \qquad \text{as }n \to \infty.
$$
Since $d_{\mathbb{H}}(w_n, w_{n+1})$ is bounded above by the Schwarz lemma, either $\re w_n \to \infty$ or $\re w_n \to -\infty$.
However, the latter case is excluded by Lemma \ref{basic-estimate-new}.
\end{proof}

For $0 < \alpha < \beta$ and $R \in \mathbb{R}$, let $S_+(R, \alpha, \beta)$ denote the right half-strip $\{ \re z > R, \, \alpha < \im z < \beta \}$. 
Following \cite{pietro}, we say that a domain $U_+$ has a {\em right lateral inner tangent} if for every $0 < \alpha < \beta$, there exists an $R = R(\alpha, \beta)$ such that $S_+(R, \alpha, \beta) \subset U_+$. Domains possessing left lateral inner tangents are defined similarly using left half-strips $S_-(R, \alpha, \beta) = \{ \re z < -R, \, \alpha < \im z < \beta \}.$

\begin{lemma}
\label{properties-of-Pplus}
The map 
$
\widetilde{P}_+(z) = \lim_{n \to \infty} F^{\circ n}(z) -  \re F^{\circ n}(i)
$
enjoys the following properties:

{\em (i)} For any $\delta > 0$,
\begin{equation}
\label{eq:forward-derivative}
\im \widetilde P_+(x+iy) - y \to 0
\qquad \text{and} \qquad
 \widetilde P'_+(x+iy) \to 1,
\end{equation}
uniformly as $x \to \infty$ and $y > \delta$.

{\em (ii)} 
$\widetilde{P}_+(z)$ is injective on the regions $\Omega_{+, \delta}$ from Lemma \ref{forward-dynamics}.

{\em (iii)} 
The domain $\widetilde{P}_+(\bigcup_{\delta > 0} \Omega_{+, \delta})$ possesses a right lateral inner tangent.

{\em (iv)} 
The quotient map
$
P_+(z) = \widetilde P_+(z) \!\! \mod T
$
onto $C_+$ is surjective. 
\end{lemma}

\begin{proof}
(i) The first statement in (\ref{eq:forward-derivative}) follows from Lemma \ref{forward-dynamics}.
By elementary complex analysis, a holomorphic function $h(z)$ defined on an open set $U \subset \mathbb{C}$ with $\im h(z) = \im z$ for $z \in U$ has the form $h(z) = z + c$ for some $c \in \mathbb{C}$, and so has constant derivative 1. From here, the second statement in (\ref{eq:forward-derivative}) follows from a normal families argument.

(ii) As the composition of injective maps is injective, each iterate $F^{\circ n}(z)$, $n \ge 1$, is injective on on $\Omega_{+, \delta}$. Consequently,
so is $F^{\circ n}(z) -  \re F^{\circ n}(i)$. By Hurwitz's theorem, the limit map $\widetilde{P}_+(z)$ is either injective or constant. Since (\ref{eq:forward-derivative}) prevents the latter possibility, $\widetilde{P}_+(z)$ is injective on  $\Omega_{+, \delta}$.

(iii) Fix $0 < \alpha < \beta$. By (i), for any $0 < \varepsilon < 1$,
we can choose $R > 0$ sufficiently large such that $|\im (\widetilde{P}_+(z) - z)| < \varepsilon$ and $|\widetilde{P}_+'(z) - 1| < \varepsilon$ on $S_+(R, \alpha, \beta)$.

For such an $R$, the image of the vertical segment $R \times [\alpha, \beta]$ is
$\varepsilon$-close to a vertical segment $R' \times [\alpha, \beta]$ in the Hausdorff metric, while the images of the horizontal  sides of $\partial S_+(R, \alpha, \beta)$ are $\varepsilon$-close to the horizontal sides of $\partial S_+(R', \alpha, \beta)$. Since $\widetilde{P}_+$ is conformal on $S_+(R, \alpha, \beta)$,
$$
S_+(R' + \varepsilon, \alpha + \varepsilon, \beta - \varepsilon) \subset \widetilde{P}_+ \bigl (S_+(R, \alpha, \beta) \bigr ).
$$
Since $\alpha, \beta, \varepsilon > 0$ were arbitrary, $\widetilde{P}_+(\Omega_{+, \delta})$ possesses a right lateral inner tangent.

(iv) follows directly from (iii).
\end{proof}

Theorem \ref{forward-cylinder-thm} now follows from Pommerenke's conjugacy theorem (Corollary \ref{pommerenke-conjuacy2}).
together with Lemma \ref{properties-of-Pplus}(iv).

 \subsection{Backward iteration}
 
As in the case of forward iteration, Lemma \ref{basic-estimate-new} implies that far to the left, $F^{-1}$ looks like translation by $-T$. We now examine the long-term behaviour of the backward iterates:

 \begin{lemma}
 \label{backward-dynamics}
For every $0 < \delta < T/3$, there exists an $X > 0$ sufficiently large such that $F$ is injective in the region
$$
\Omega_{-, X, \delta} \, = \, \{ x + iy \, : \, x < -X, \, y  > \delta \}
$$
and admits an inverse branch which satisfies
\begin{equation}
\label{eq:backward-iteration-is-nice}
|F^{-1}(z) - z - T| < \delta, \qquad z \in \Omega_{-, X, \delta}.
\end{equation}
Moreover, for any $z \in \Omega_{-, X, 2\delta}$, the backward orbit $F^{-n}(z)$ is contained in $\Omega_{-, X, \delta}$
and satisfies
\begin{equation}  
\label{eq:backward-iteration-is-nice2}
0 \, < \, \im z - \im F^{-n}(z) \, < \, \delta, \qquad n \ge 1.
\end{equation}
\end{lemma}

The proof is similar to that of Lemma \ref{forward-dynamics}.

\begin{lemma}
Any inverse orbit ${\bf z} = (z_{-n})_{n=0}^\infty$ of positive height is eventually contained in a region $\Omega_{-, X, \delta}$ satisfying the conclusion of Lemma \ref{backward-dynamics}. More precisely,
there exists an $n_0 > 0$ and $\delta > 0$ such that 
$$
z_{-n} \in \Omega_{-, X, \delta}, \qquad F^{-1}(z_{-n}) = z_{-n-1}, \qquad n \ge n_0,
$$
where $F^{-1}$ denotes the natural inverse branch.
\end{lemma}

The proof is similar to that of Lemma \ref{forward-dynamics2}.

\begin{lemma}
\label{properties-of-Pminus}
Fix an inverse orbit ${\bf c} = (c_{-n})_{n=0}^\infty$ of positive height, whose existence is guaranteed by Corollary \ref{PS-existence-of-bi-infinite-orbits}.
The limit
\begin{equation*}
\widetilde{P}_-(z) = \lim_{n \to \infty} \bigl ( F^{-n}(z) - \re c_{-n} \bigr )
\end{equation*}
exists and defines a holomorphic function on $\Omega_{-, X, 2\delta}$ which satisfies the functional equation
$$
\widetilde{P}_-(F(z)) = \widetilde{P}_-(z) + T.
$$
and enjoys the following properties:

{\em (i)} For any $\delta > 0$,
\begin{equation}
\label{eq:back-derivative}
\im \widetilde P_-(x+iy) - y \to 0
\qquad \text{and} \qquad
 \widetilde P'_-(x+iy) \to 1,
\end{equation}
uniformly as $x \to -\infty$ and $y > 2\delta$.

{\em (ii)} 
$\widetilde{P}_-(z)$ is injective on the regions $\Omega_{-, X, 2\delta}$ from Lemma \ref{backward-dynamics}.

{\em (iii)} 
The domain $\widetilde{P}_-(\bigcup_{\delta > 0} \Omega_{-, X, 2\delta})$ possesses a left lateral inner tangent.

{\em (iv)} 
The quotient map
$
P_-(z) = \widetilde P_-(z) \!\! \mod T
$
onto $C_-$ is surjective. 
\end{lemma}

The proof is similar to that of Lemma \ref{properties-of-Pplus}.
Theorem \ref{backward-dynamics3} follows from Poggi-Corradini's conjugacy theorem (Corollary \ref{pietro-conjuacy2}) together with Lemma \ref{properties-of-Pminus}(iv).

\subsection{\texorpdfstring{Parabolicity of the maps $\widetilde{P}_+$ and $\Psi$}{Parabolicity of the maps P̃₊ and Ψ}}
\label{sec:checking-parabolicity}

Recall from Corollary \ref{pommerenke-conjuacy2} that
$$
\widetilde{P}_+(z) = \lim_{n \to \infty} \bigl \{ F^{\circ n}(z) - \re F^{\circ n}(i) \bigr \}.
$$
In Section \ref{sec:imaginary-forward}, we saw that for a point $z \in \mathbb{H}$ with $\im z > Y$,
$$\im z \, < \, \im F^{\circ n}(z) \, < \, C + \im z, \qquad n = 0, 1, 2, \dots.$$  
As the orbit $F^{\circ n}(i)$ has bounded height, we have
$$1 \, < \, \im F^{\circ n}(i) \, < \, C' + 1, \qquad n = 0, 1, 2, \dots,$$
for some constant $C' > 0$. Consequently, $$\frac{\im \widetilde{P}_+(z)}{\im z} \to 1, \qquad \text{as } \im z \to \infty.$$
 By Julia's lemma, $\widetilde{P}_+(\infty) = \infty$ with angular derivative 1. A similar argument shows that 
 Poggi-Corradini's conjugacy
  $$
 \Psi = \lim_{n \to \infty} F^{\circ n}(z + \re c_{-n})
 $$
 from Corollary \ref{pietro-conjuacy2} is parabolic.

\section{Lavaurs semigroups}
\label{sec:lavaurs}

In this section, we associate forward and backward quotient half-cylinders $C_+, C_- \cong  \mathbb{H}/(z \to z + T)$ to a simply parabolic Polya-Szeg\"o inner function $F$. Given $\sigma \in \mathbb{R}$, we construct a commuting holomorphic map $G_\sigma$ and define a semigroup $\mathcal G_\sigma$.
The map $G_\sigma$ induces a holomorphic self-map $L_\sigma$ of the forward half-cylinder $C_+$, which we call the {\em Lavaurs map}\/. We show that $L_\sigma$ extends to a holomorphic self-map of the unit disk with a Denjoy-Wolff fixed point at the origin.
The main results of this section are summarized
in Theorem \ref{summary1} in the introduction.

\subsection{The Lavaurs semigroup}

Without loss of generality, we assume that $T > 0$.
Recall from Theorem \ref{forward-cylinder-thm}, that
$$
\widetilde{P}_+(z) = \lim_{n \to \infty} F^{\circ n}(z) - \re F^{\circ n}(i)
$$
defines a holomorphic function on $\mathbb{H}$ which satisfies the functional equation $\widetilde{P}_+(F(z)) = \widetilde{P}_+(z) + T$.
Similarly, by Lemma \ref{properties-of-Pminus}, the map 
$$
\widetilde{P}_-(z) = \lim_{n \to \infty} F^{-n}(z) - \re c_{-n}
$$
is defined on a domain $\Omega_- \subset \mathbb{H}$ where it is injective and satisfies
$\widetilde{P}_-(F(z)) = \widetilde{P}_-(z) + T.$
 Moreover, for every $\alpha, \beta > 0$, there exists an $R = R(\alpha, \beta) > 0$ such that 
$$
\bigl \{ z \in \mathbb{H} : \re z < -R, \, \alpha < \im z < \beta \bigr \}  \, \subset \, \widetilde{P}_-(\Omega_-).
$$

Fix $z \in \mathbb{H}$. For $\sigma <\!\!< 0$, the translate $\widetilde{P}_+(z) + \sigma$ lies in $\widetilde{P}_-(\Omega_-)$, so we may define
$$
G_\sigma(z) := \widetilde{P}_-^{-1} ( \widetilde{P}_+(z) + \sigma  ).
$$
For a general $\sigma \in \mathbb{R}$, we choose an integer $k >\!\!> 0$ sufficiently large so that
$\widetilde{P}_+(z) + \sigma - kT \in \widetilde{P}_-(\Omega_-)$ and define
$$
G_\sigma(z) := F^{\circ k}(G_{\sigma - kT}(z))
$$
This definition is independent of the choice of $k$ since iteration by $F$ corresponds to translation by $T$ in the coordinate $\widetilde{P}_+$. From this construction, it is immediate that $G_\sigma$ is a holomorphic self-map of $\mathbb{H}$ which commutes with $F$ and the semigroup $\mathcal G_\sigma$ constructed in Section \ref{sec:lavaurs-intro} depends only on $\sigma  \!\! \mod T$.

\subsection{The Lavaurs map}
\label{sec:the-lavaurs-map}

We now construct the Lavaurs map. Recall that by Theorem \ref{backward-dynamics3}, for any point $w \in C_- \cong  \mathbb{H}/(z \to z + T)$, there is a bi-infinite orbit
${\bf z} = (z_n)_{n=-\infty}^\infty$ with
$$
\lim_{n \to -\infty} (z_{n+1} - z_n) = T, \qquad \lim_{n \to \-\infty} z_n \!\!\mod T = w,
$$
which is unique up to replacing $(z_n)_{n=-\infty}^\infty$ with  $(z_{n+k})_{n=-\infty}^\infty$ for some $k \in \mathbb{Z}$. We define $F^\infty: C_- \to C_+$ as the map which takes
$w = P_-({\bf z})$ to $P_+({\bf z})$.

Let $\tau_\sigma: C_+ \to C_-$ be the translation  $w \to w + \sigma \!\! \mod T$. The Lavaurs map $L_\sigma: C_+ \to C_+$ is then defined as the composition $F^\infty \circ \tau_\sigma$. By construction, $F^\infty$ and $L_\sigma$ are holomorphic and $L_\sigma$ satisfies $L_\sigma(P_+(z)) = P_+(G_\sigma(z))$. In other words, $L_\sigma$ is the quotient of $G_\sigma$ under the forward projection.

We now pass to the punctured unit disk model.
The exponential map $z \to e^{2\pi i z/T}$ identifies $C_+$ conformally with $\mathbb{D}^*$. Under this identification, $L_\sigma$ becomes a holomorphic self-map of the punctured disk.
As the Lavaurs map $L_\sigma: C_+ \to C_+$ increases
the imaginary part in the cylinder model, $L_\sigma: \mathbb{D}^* \to \mathbb{D}^*$ decreases the absolute value in the punctured disk model. It follows that the origin is a removable singularity and the extension $L_\sigma: \mathbb{D} \to \mathbb{D}$ 
has an attracting fixed point at the origin. To compute its multiplier, note that
\begin{equation*}
\lim_{y \to \infty} \im F^\infty(x+iy) - y = (\pi/T) \mu(\mathbb{R})
\end{equation*}
 by Lemma \ref{large-imaginary-part}, which translates to
$$
|L_\sigma'(0)| = \exp \bigl ( -2\pi^2 \cdot \mu(\mathbb{R})/T^2 \bigr ).
$$
Since the multiplier is strictly between 0 and 1, the origin is an attracting but not super-attracting fixed point and the quotient $\mathbb{H} / \mathcal G_\sigma \cong \mathbb{D}^* / L_\sigma$ is a torus.

Finally, Lemma \ref{large-imaginary-part}, also yields the estimate
$$\im z + 2C \, > \, \im G^{\circ 2}(z) \, > \, \im z + C,$$
 for $z \in \mathbb{H}$ with sufficiently large imaginary part.
As in Section \ref{sec:checking-parabolicity}, this implies that that $G_\sigma$ is a parabolic inner function of infinite height.

\section{Inner functions with finite Lyapunov exponent}
\label{sec:FLE}

In this section, we study simply parabolic inner functions with finite Lyapunov exponent. The results are summarized in Theorem \ref{summary2}.

\subsection{Finite derivative along orbits}

We denote by $S_{\pi/2}(x) = \bigl \{ z \in \mathbb{H} : \im z > |\re (z- x)| \bigr \}$ the Stolz angle with vertex at $x \in \mathbb{R}$ and opening $\pi/2$. We write $S_{\pi/2}(x, h) := S_{\pi/2}(x) \cap \{ \im z < h \}$ for the Stolz angle truncated at height $h$. 

\begin{lemma}
\label{finite-log-sum}
Suppose $F: \mathbb{H} \to \mathbb{H}$ is a simply parabolic inner function with finite Lyapunov exponent. 

{\em (i)} If $(x_n)_{n=0}^\infty$ is a forward orbit on the real line such that
$$
\sum_{n=0}^\infty \log F'(x_n) < \infty,
$$
then
$
\widetilde{P}_+(z)
$
has a finite angular derivative at $x_0$ with
$$
\widetilde{P}_+(x_0) \, := \, \lim_{y \to 0} \widetilde{P}_+(x_0 + iy) \, = \, \lim_{n \to \infty} (x_{n} - \re c_{n}).
$$ 
and
$$
\widetilde{P}'_+(x_0) \, := \, \lim_{y \to 0} \widetilde{P}'_+(x_0 + iy) \, = \, \prod_{n=0}^\infty F'(x_n).
$$ 

{\em (ii)} If ${\bf x} = (x_{-n})_{n=0}^\infty$ is an inverse orbit on the real line such that
$$
\sum_{n=1}^\infty \log F'(x_{-n}) < \infty,
$$
then there exists a truncated Stolz angle $S = S_{\pi/2}(x_0, h)$ of opening $\pi/2$ with a vertex at $x_0$, and inverse branches
$(F^{-n})_{n=0}^\infty$, defined on $S$ with
$$
\lim_{y \to 0} \, F^{-n}(x_0 + iy) = x_{-n}, \qquad n = 0, 1, 2, \dots.
$$
Furthermore, the function
$$
\widetilde{P}_-(z) = \lim_{n \to \infty} \bigl ( F^{-n}(z) - \re c_{-n} \bigr ), \qquad z \in S,
$$
satisfies
 $$
 \widetilde{P}_-({\bf x}) \, := \, \lim_{y \to 0} \widetilde{P}_-(x_0 + iy) \, = \,  \lim_{n \to \infty} (x_{-n} - \re c_{-n})
$$
and
$$
 (\widetilde{P}_-)'({\bf x}) \, := \, \lim_{y \to 0} \, (\widetilde{P}_-)'(x_0 + iy) \, = \, \biggl (\prod_{n=1}^\infty F'(x_{-n}) \biggr )^{-1}.
 $$
As a result, the inverse mapping $\Psi: \mathbb{H} \to \mathbb{H}$ has a finite angular derivative at $\widetilde P_-({\bf x})$ with
$$\Psi'(\widetilde P_-({\bf x})) = \prod_{n=1}^\infty F'(x_{-n}).$$

{\em (iii)}
 If ${\bf x} = (x_{n})_{n=-\infty}^\infty$ is a bi-infinite orbit on the real line such that
$$
\sum_{n=-\infty}^\infty \log F'(x_{n}) < \infty,
$$
then $F^\infty: C_- \to C_+$ maps $P_-({\bf x})$ to $P_+({\bf x})$ and satisfies 
$$
(F^\infty)'(P_-({\bf x})) \, := \, (F^\infty)'(P_-({\bf x})) \, = \, \prod_{n=-\infty}^\infty F'(x_{n}).
$$
\end{lemma}

The proof relies on the following elementary lemma:

\begin{lemma}
\label{ad-estimate}
Suppose $F: \mathbb{H} \to \mathbb{H}$ is a parabolic inner function. If $F(0) = 0$ and $F'(0) < \infty$, then
$$
|F'(z) - 1| < 2(F'(0) - 1),
$$
for any $z$ in the Stole angle $S_{\pi/2}(0) = \{ z \in \mathbb{H} : \im z > |\re z| \}$.
\end{lemma}

Before giving the proof, we mention two simple consequences of the above lemma. By the fundamental theorem of calculus, 
$$
F(z_2) - F(z_1) = (z_2 - z_1) + \int_{z_1}^{z_2} (F'(w) - 1) dw, \qquad z_1, z_2 \in S_{\pi/2}(0).
$$
Therefore, if $F'(0) < 3/2$, then $F$ is injective in the Stolz angle $S_{\pi/2}(0)$. Taking $z_2 = z$, $z_1 = i \varepsilon$ and taking $\varepsilon \to 0$, we get
$$
\biggl | \frac{F(z)}{z} - 1 \biggr | \le 2(F'(0) - 1), \qquad z \in S_{\pi/2}(0).
$$

\begin{proof}
Suppose $F$ has the representation (\ref{eq:general-inner}). Recall from (\ref{eq:derivative-general}) that
$$
F'(z) - 1= \int_{\mathbb{R}} \frac{1+a^2}{(z - a)^2} \, d\nu(a).
$$
Since $F$ has a finite angular derivative at zero, $\nu$ does not have a point mass at zero.
The lemma follows as
$|0 - a| < \sqrt{2} \, |z - a|$ for any $a \in \mathbb{R} \setminus \{ 0 \}$ and $z \in S_{\pi/2}(0)$.
\end{proof}

\begin{proof}[Proof of Lemma \ref{finite-log-sum}]
(i) {\em Step 1. Identification of the boundary value.} In this step, we show that $x_n - \re c_n$ converges as $n \to \infty$, and that the limit agrees with the vertical boundary value of $\widetilde{P}_+(z)$ at $x_0$. For $y > 0$ and $n = 0, 1, 2, \dots$, we may write
$$
(x_n - \re c_n) = (F^{\circ n}(x_0) - F^{\circ n}(x_0 + iy)) + (F^{\circ n}(x_0 + iy) - \re c_n).
$$
Since the second term on the right converges to $\widetilde{P}_+(x_0 + iy)$ by Corollary \ref{pommerenke-conjuacy2}, it suffices to show that the first term tends to 0 as $y \to 0$, uniformly in $n$. This follows from Lemma \ref{ad-estimate} and the fundamental theorem of calculus:
\begin{align*}
|F^{\circ n}(x_0) - F^{\circ n}(x_0 + iy)| & \le y \cdot \max_{z \in S_{\pi/2}(x_0)} |(F^{\circ n})'(z)| \\
& \le y \cdot \biggl (2 \prod_{k=0}^\infty F'(x_{k}) -1 \biggr ).
\end{align*}

\medskip

{\em Step 2. Existence of the angular derivative.}
We now turn to showing the existence of the angular derivative $\widetilde{P}'_+(x_0) := \lim_{y \to 0} \widetilde{P}'_+(x_0 + iy)$. Given $\varepsilon > 0$, choose $N = N(\varepsilon) > 0$ sufficiently large such that
\begin{equation}
\label{eq:derivative-tail-estimate}
\sum_{k = N}^\infty \log F'(x_k) < \varepsilon.
\end{equation}
Since $F^{\circ N}$ has an angular derivative at $x_0$, for $0 < y < y_0(\varepsilon, N)$ sufficiently small, we have
 $F^{\circ N}(x_0 + iy) \in S_{\pi/2}(x_N)$ and
\begin{equation}
\label{eq:important-ratio}
1 - O(\varepsilon) \, < \, \frac{(F^{\circ N})'(x_0 + iy)}{(F^{\circ N})'(x_0)} \, < \, 1 + O(\varepsilon).
\end{equation}
For $n \ge N$, we factor
$$
(F^{\circ n})'(x_0+iy) = (F^{\circ N})'(x_0+iy) \cdot (F^{\circ (n-N)})'(F^{\circ N}(x_0 + iy)).
$$

By (\ref{eq:important-ratio}) and (\ref{eq:derivative-tail-estimate}),
\begin{align*}
(F^{\circ N})'(x_0+iy) & = (1+O(\varepsilon)) \cdot (F^{\circ N})'(x_0) \\
& = (1+O(\varepsilon)) \prod_{k=0}^{N-1} F'(x_k) \\ 
& =  (1+O(\varepsilon)) \prod_{k=0}^\infty F'(x_k).
\end{align*}
Similarly, by Lemma \ref{ad-estimate} and (\ref{eq:derivative-tail-estimate}),
$$
(F^{\circ (n-N)})'(F^{\circ N}(x_0 + iy)) \, = \, (1+O(\varepsilon)) \cdot (F^{\circ (n-N)})'(F^{\circ N}(x_0))
\, = \, 1+O(\varepsilon).
$$
Combining the above observations, we see that for any $\varepsilon > 0$, when $n \ge N(\varepsilon)$ is sufficiently large and $y < y_0(\varepsilon)$ is sufficiently small,
$$
(1 - O(\varepsilon)) \prod_{k=0}^\infty F'(x_k) \, < \, (F^{\circ n})'(x_0+iy) \, < \, (1+O(\varepsilon)) \prod_{k=0}^\infty F'(x_k),
$$
or alternatively,
$$
(1 - O(\varepsilon))  \prod_{k=0}^\infty F'(x_k) \, < \,  F_n'(x_0 + iy) \, < \, (1+O(\varepsilon))  \prod_{k=0}^\infty F'(x_k).
$$
where $F_n(z) := F^{\circ n}(z + x_0) - \re c_n$. Taking the limit as $n \to \infty$, we see that the above inequality holds with $\widetilde{P}_+$ in place of $F_n$\,:
$$
(1 - O(\varepsilon))  \prod_{k=0}^\infty F'(x_k) \, < \, \widetilde P_+'(x_0 + iy) \, < \, (1+O(\varepsilon))  \prod_{k=0}^\infty F'(x_k).
$$
Since $\varepsilon > 0$ was arbitrary, $\lim_{y \to 0} \widetilde P'_+(x_0 + iy) = \prod_{k=0}^\infty F'(x_k)$ as desired.

\medskip

The proof of (ii) is similar,  while (iii) is an immediate consequence of (i) and (ii).
\end{proof}

The following lemma due to M.~Heins \cite{heins77} identifies the angular derivative as the measure-theoretic Jacobian:

\begin{lemma}
\label{jacobian-lemma}
Suppose $\varphi: \mathbb{H} \to \mathbb{H}$ is a holomorphic mapping which has an angular derivative at each point of $A$. If the vertical extension of $\varphi$ is injective on $A$, then
$$
\ell(\varphi(A)) = \int_A |\varphi'(x)| d\ell(x).
$$
\end{lemma}
 
For a proof via almost uniform differentiability, see \cite[Lemma 1.6]{craizer}. Below, we will apply this to $\varphi = \Phi$.

\subsection{Wandering sets}

By a {\em wandering set} $A \subset \mathbb{R}$, we mean a measurable set which does not intersect any of its images under $F$, i.e.~
$$
A \cap F^{\circ n}(A) = \emptyset, \qquad \text{for any }n > 0.$$
It is easy to see that this condition implies that the inverse images of $A$ are pairwise disjoint: 
$$
F^{-m}(A) \cap F^{-n}(A) = \emptyset, \qquad \text{for any }n > m \ge 0.
$$
Since under iteration, a.e.~$x \in \mathbb{R}$ converges to infinity by Lemma \ref{convergence-lemma}, any bounded set 
$A \subset \mathbb{R}$ of positive Lebesgue measure contains a wandering subset of positive measure, for instance, one can take $A_0 = A \setminus \bigcup_{n \ge 1} F^{\circ n}(A)$. From here, it follows that any set $A \subset \mathbb{R}$ of finite Lebesgue measure can be decomposed into a union of countably many basic sets and a set of measure zero.

\begin{lemma}
\label{T-minus-epsilon}
Suppose $F: \mathbb{H} \to \mathbb{H}$ is a simply parabolic inner function with finite Lyapunov exponent. For any $\varepsilon > 0$, there exists a wandering set $A \subset \mathbb{R}$ with $\ell(A) \ge T - \varepsilon$.
\end{lemma}

The proof relies on the following lemma:

\begin{lemma}
\label{thin-sets}
Let $F: \mathbb{H} \to \mathbb{H}$ be a parabolic inner function with a Denjoy-Wolff fixed point infinity. Then $F$ is injective on the set $\mathbb{R}_{\thin} = \{ x \in \mathbb{R} : |F'(x)| < 2 \}$.
\end{lemma}

When $F: \mathbb{H} \to \mathbb{H}$ is a finite Blaschke product, the lemma follows from the invariance the Lebesgue measure on the real line: since $\sum_{F(\zeta) = \xi} |F'(\zeta)|^{-1} = 1$ for any $\xi \in \mathbb{R}$, there cannot be two points in the thin set which have the same image under $F$. The general case is similar, but uses Aleksandrov-Clark measures. The proof will be given in Appendix  \ref{sec:ac-measures}.

\begin{remark}
The thick-thin decomposition was introduced by C.~McMullen \cite[Section 4]{mcmullen2} for finite Blaschke products with $F(0) = 0$. There, McMullen showed that the restriction of $F : 
\partial \mathbb{D} \to \partial \mathbb{D}$ to the thin part of $F$ extends to a homeomorphism of the unit circle.
\end{remark}

\begin{proof}[Proof of Lemma \ref{T-minus-epsilon}]
Fix an $0 < \varepsilon < 2T$. By the Cauchy-Schwarz inequality,
$$
\int_{\mathbb{R}} \frac{d\mu(a)}{|x-a|} \, \le \, \biggl ( \int_{\mathbb{R}} d\mu(a) \biggr)^{1/2} \biggl ( \int_{\mathbb{R}} \frac{d\mu(a)}{|x-a|^2} \biggr)^{1/2}
\, = \, \mu(\mathbb{R})^{1/2} ( |F'(x)| - 1 )^{1/2}.
$$
Since $\mu(\mathbb{R})$ and the Lyapunov exponent of $F$ is finite, the sets
$$
\bigl \{ x \in \mathbb{R} : |F(x) - x - T| < \varepsilon/2 \bigr \} \qquad \text{and} \qquad  \mathbb{R}^c_{\thin} = \{ x \in \mathbb{R} : |F'(x)| \ge 2 \}
$$ 
have finite Lebesgue measure. Let 
$$
B = \bigl \{ x \in \mathbb{R} : |F(x) - x - T| < \varepsilon/2 \bigr \} \cup \mathbb{R}^c_{\thin}
$$ 
be their union and choose $t > 0$ sufficiently large so that $\ell(B \cap [t, \infty)) < \varepsilon/2$. Define 
$$A \subset [t+\varepsilon/2, t + T)$$
as the set of points that do not belong to $B$ and whose forward orbit never enters $B$. 
Finally, for each $n \ge 0$, let $E_n$ be the set of points $x \in [t+\varepsilon/2, t+T)$ for which $F^{\circ n}(x) \in B$
but $F^{\circ k}(x) \notin B$ for $0 \le k < n$. Then, $[t+\varepsilon/2, t+T) = A \sqcup \bigsqcup_{n \ge 0} E_n$.

Since $$T - \varepsilon/2 \, \le \, F^{\circ (n+1)}(x) - F^{\circ n}(x) \, \le \, T+\varepsilon/2, \qquad x \in A, \quad n \ge 0,$$
the set $A$ is disjoint from its forward iterates. As $F$ is injective on the thin set of $F$, the sets $B_n = F^{\circ n}(E_n) \subset B$, $n \ge 0$, are disjoint. Since $\ell$ is $F$-invariant, $\ell(E_n) \le  \ell(B_n)$ and so
$$
\sum_{n=0}^\infty \ell(E_n) \, \le \, \sum_{n=0}^\infty \ell(B_n) \, \le \, \varepsilon/2.
$$
Hence, $\ell(A) \ge T - \varepsilon$ as desired.
\end{proof}

\subsection{\texorpdfstring{Correspondence between $\partial C_-$ and $\partial C_+$}{Correspondence between ∂C₋ and ∂C₊}}
  
Previously in Section \ref{sec:the-lavaurs-map}, for a simply parabolic Polya-Szeg\"o inner function, we constructed a holomorphic map $F^\infty: C_- \to C_+$ by considering BISBS in the upper half-plane. More precisely, if ${\bf z} = (z_n)_{n=-\infty}^\infty$ is a BISBS, then
$F^\infty$ maps $P_-({\bf z}) \in C_-$ to $P_+({\bf z}) \in C_+$. 
The purpose of this section is to understand the boundary behaviour of this correspondence when $F$ has a finite Lyapunov exponent. 
Recall that by Lemma \ref{finite-log-sum} above, if ${\bf x} = (x_{-n})_{n=0}^\infty \in \widehat{\mathbb{R}}$ is a bi-infinite orbit with
\begin{equation}
\label{eq:corr-log-sum}
\sum_{n=-\infty}^\infty \log |F'(x_n)| < \infty,
\end{equation}
then $F^\infty$ maps $P_-({\bf x}) \in \partial C_-$ to $P_+({\bf x}) \in \partial C_+$. We denote by $\LS$ the set of bi-infinite orbits satisfying
(\ref{eq:corr-log-sum}).

We first show that $\widehat{\ell}$ a.e.~bi-infinite orbit ${\bf x} \in \widehat{\mathbb{R}}$ satisfies this summability condition. We then show that these ``good'' orbits are abundant in the sense that they cover $\partial C_-$ up to a set of Lebesgue measure zero, allowing us to identify $\partial C_-$ with
 $\widehat{\mathbb{R}}/\langle \widehat{F} \rangle$ up to null sets. 
Finally, we show that under this identification, the Lebesgue measure on
 $\partial C_- \cong \mathbb{R} / (z \to z+T)$ corresponds to the quotient of the natural extension measure $\widehat{\ell}$ on 
  $\widehat{\mathbb{R}}/\langle \widehat{F} \rangle$. We conclude that $F^\infty: C_- \to C_+$ is an inner function.

\begin{lemma}
\label{corr-log-sum}
For $\widehat{\ell}$ a.e.~bi-infinite orbit ${\bf x} = (x_{-n})_{n=0}^\infty \in \widehat{\mathbb{R}}$, the sum
\begin{equation}
\label{eq:corr-log-sum2}
\sum_{n=-\infty}^\infty \log |F'(x_n)| < \infty.
\end{equation}
In particular, for $\ell$ a.e.~$x \in \mathbb{R}$, the sum over its forward orbit 
$$
\sum_{n=0}^\infty \log |F'(x_n)| < \infty.
$$
\end{lemma}

\begin{proof}
 {\em Step 1.} Let $A \subset \mathbb{R}$ be a wandering set and $\widehat{A} \subset \widehat{\mathbb{R}}$ be the set of bi-inverse orbits ${\bf x} = (x_{n})_{n=-\infty}^\infty$ with $x_0 \in A$. From the definition of the natural extension, it follows that if a function on $\widehat{\mathbb{R}}$ only depends on a single coordinate, say $\phi({\bf x}) = \phi_k(x_k)$, then integrating over $\widehat{\mathbb{R}}$ reduces to integration on the base space:
$$
\int_{\widehat{\mathbb{R}}} \phi({\bf x}) \, d\widehat{\ell}({\bf x}) = \int_{\mathbb{R}} \phi_k (x_k) \, d\ell(x_k).
$$
Suppose $k \ge 1$. Since $\widehat{A} = \bigl \{ {\bf x} \in \widehat{\mathbb{R}} : x_k \in F^{-k}(A) \bigr \}$, replacing $\phi$ by $\phi \cdot \chi_{\widehat{A}}$ gives
$$
\int_{\widehat{A}} \phi({\bf x}) \, d\widehat{\ell}({\bf x}) = \int_{F^{-k}(A)} \phi_k (x_k) \, d\ell(x_k).
$$
As the sets $\{ F^{-k}(A) : k \ge -1 \}$ are disjoint,
$$
\int_{\widehat{A}} \biggl \{ \sum_{k=-\infty}^{-1} \log |F'(x_{k})| \biggr \} d\widehat{\ell} \, = \, 
  \sum_{k=-\infty}^{-1} \int_{F^{\circ k}(A)} \log |F'(x)| d\ell 
\, \le \, \int_{\mathbb{R}} \log |F'(x)| d\ell.
$$
The interchange of summation and integration can be justified by the monotone convergence theorem since all terms are non-negative.

\medskip

\noindent {\em Step 2.}
By the change of variables formula,
$$
\int_A \biggl \{ \sum_{n = 0}^\infty \log |F'(x_n)|\biggr \} d\ell(x)
=
\int_{\mathbb{R}} \log |F'(\xi)| \cdot W(\xi) \, d\ell(\xi),
$$
where
$$
W(\xi) = \sum_{\substack{n \ge 0, \, a \in A \\ F^{n(a)}(a)=\xi}}
\frac{1}{|(F^{n(a)})'(a)|}.
$$
It remains to show that $W(\xi) \le 1$ for almost every $\xi \in \mathbb{R}$, which will imply
$$
\int_A \biggl \{ \sum_{n = 0}^\infty \log |F'(x_n)|\biggr \} d\ell(x)
\, \le \,
\int_{\mathbb{R}} \log |F'(\xi)|\, d\ell(\xi) \, < \, \infty.
$$

Fix a point $\xi \in \mathbb{R}$ and consider the tree $T_\xi$ of pre-images of $\xi$. The vertices of $T_\xi$ are points $\zeta \in \mathbb{R}$ that map to $\xi$ under some iterate of $F$. We view $\xi$ as the root vertex of $T_\xi$.
 Two vertices $\zeta$ and $\eta$ are connected by an edge if $F(\zeta) = \eta$ or vice versa.

Since Lebesgue measure is invariant under $F$, for a.e.~$b \in \mathbb{R}$, we have
\begin{equation}
\label{eq:preimage-1-sum}
\sum_{F(a)=b} \frac{1}{|F'(a)|} = 1.
\end{equation}
By excluding a set of $\xi$'s of measure zero, we may assume that $F$ has an angular derivative at every vertex $\zeta \in T_\xi$ and that (\ref{eq:preimage-1-sum}) holds with $b=\zeta$. To a vertex $\zeta \in T_v$, we assign the weight
$$
w(\zeta):=\frac{1}{|(F^{\circ n(\zeta)})'(\zeta)|}, \qquad \text{where }F^{\circ n(\zeta)}(\zeta)=\xi.
$$
From (\ref{eq:preimage-1-sum}) and the chain rule, it follows that the weight of any vertex is equal to the sum of the weights of its children.

Since $A$ is a wandering set, no point in $A$ can be a descendant of any other point of $A$. (In other words, vertices in $A$ form an anti-chain). Since the weight is conserved down the tree, the total weight carried by any anti-chain is bounded above by the weight of the root. As the root has weight 1, we have $W(\xi) \le 1$.

\medskip

\noindent {\em Conclusion.}
Putting the estimates in Steps 1 and 2 together, we get
$$
\int_{\widehat{A}} \biggl \{ \sum_{n=-\infty}^\infty \log |F'(x_{k})| \biggr \} d\widehat{\ell} \, \le \, 2 \int_{\mathbb{R}} \log |F'(x)| d\ell(x) \, < \, \infty. 
$$
(A careful inspection of the above argument shows that the factor of 2 is not necessary.)
Consequently, for $\widehat{\ell}$ a.e.~bi-infinite orbit ${\bf x} \in \widehat{A}$, the sum
$$
\sum_{n=-\infty}^\infty \log |F'(x_n)| < \infty.
$$
Since the real line can be expressed as a countable union of wandering sets and a null set, the above sum is finite for  $\widehat{\ell}$ a.e.~${\bf x} \in \widehat{\mathbb{R}}$.
\end{proof}

By Lemma \ref{finite-log-sum}, we have:

\begin{lemma}
For $\widehat{\ell}$ a.e.~bi-infinite orbit ${\bf x} = (x_{-n})_{n=-\infty}^\infty \in \widehat{\mathbb{R}}$, the limits
$$
\widetilde{P}_+({\bf x}) := \lim_{n \to \infty} \bigl ( x_n - \re F^{\circ n}(i) \bigr ), \qquad \widetilde{P}_-({\bf x}) := \lim_{n \to \infty}  \bigl (  x_{-n}  - \re c_{-n} \bigr ) $$ exist. Moreover, $F^{\infty}(P_-({\bf x})) = P_+({\bf x})$
and \,$\log |(F^{\infty})'(P_-({\bf x}))| = \sum_{n=-\infty}^\infty \log |F'(x_n)|.$
\end{lemma}

\begin{lemma}
\label{wandering-equality}
For any wandering set $A \subset \mathbb{R}$,
$$
\ell(P_-(\widehat{A} \cap \LS)) = \ell(A).
$$
\end{lemma}

\begin{proof}
 Since no two inverse orbits in $\widehat{A}$ are related by an iterate of $\widehat{F}$, no two points of $\widetilde{P}_-(\widehat{A} \cap \LS)$ differ by a multiple of $T$. Consequently,
$$
\ell \bigl (P_-(\widehat{A} \cap \LS) \bigr ) \, = \, \ell \bigl (\widetilde{P}_-(\widehat{A} \cap \LS) \bigr) \, \le \, \ell(\Psi^{-1}(A)).
$$ 
Since $\Psi$ is a self-map of the upper half-plane with a parabolic Denjoy-Wolff point at infinity, L\"owner's lemma (see Appendix B for a proof using Aleksandrov-Clark measures) tells us that
$\ell(\Psi^{-1}(A)) \le \ell(A)$. Therefore,
$$
\ell \bigl (P_-(\widehat{A} \cap \LS) \bigr ) \le \ell(A).
$$

It remains to prove the opposite inequality
\begin{equation}
\label{eq:corr-lower-bound}
\ell \bigl (P_-(\widehat{A} \cap \LS) \bigr ) \ge \ell(A).
\end{equation}
Fix an $0 < \varepsilon < \log 2$. By Lemma \ref{corr-log-sum},  there exists an $N \ge 1$ sufficiently large such that the set $G \subset \widehat{A}$ of inverse orbits which satisfy
$$
\sum_{n = -\infty}^{-N-1} \log |F'(x_n)| < \varepsilon
$$
has measure $\widehat{\ell}(G) > \ell(\widehat{A}) - \varepsilon = \ell(A) - \varepsilon$. Let $G_{-n} = \pi_{-n}(G)$ be the projection of $G$ onto the $(-n)$-th coordinate. By Lemma \ref{thin-sets}, $F$ maps $G_{-n-1}$ to $G_{-n}$ injectively for all $n \ge N$. Since $F$ is expanding on the real line, i.e.~$|F'(x)| > 1$, the sequence $\{ \ell(G_{-n}) \}_{n \ge N}$ is decreasing.
Since $G = \bigcap_{n \ge 0} \widehat{F}^{\circ n} (\widehat{G_{-n}})$, by the definition of the natural extension measure, we have $\widehat{\ell}(G) = \lim_{n \to \infty} \ell(G_{-n})$.

 Let $G' = G \cap \LS$ and $G'_{-n} = \pi_{-n}(G')$ be the projection of $G'$ onto the $(-n)$-th coordinate. This operation changes each $G_{-n}$ by an asymptotically negligible amount, i.e.~$\ell(G_{-n} \setminus G'_{-n}) \to 0$ as $n \to \infty$.
 
Let $E_{-n} = \widetilde{P}_-(G') - nT \subset \mathbb{R}$. By construction, for $n \ge N$, the Poggi-Corradini conjugacy $\Psi$ maps $E_{-n}$ bijectively onto $G'_{-n}$.
Furthermore, by Lemma \ref{finite-log-sum}, at each point of $E_{-n}$, the angular derivative of $\Psi$ is between $1$ and $e^{\varepsilon}$.
Applying Lemma \ref{jacobian-lemma} shows that for all $n \ge N$ sufficiently large,
$$
\ell \bigl (P_-(\widehat{A} \cap \LS) \bigr ) \, \ge \, \ell ( \widetilde{P}_-(G')) \, = \,
\ell(E_{-n}) \, \ge \, e^{-\varepsilon} \ell(G'_{-n}) \, \ge \,  e^{-\varepsilon} (\ell(A) - 2\varepsilon).
$$
The inequality (\ref{eq:corr-lower-bound}) follows after taking $\varepsilon \to 0$.
\end{proof}

\begin{lemma}
\label{c-minus-correspondence}
The following statements hold:

{\em (i)} $P_-(\LS)$ is a full measure subset of $\partial C_-$.

{\em (ii)} The map $F^\infty: C_- \to C_+$ is an inner function.

{\em (iii)} Under the correspondence ${\bf x} \to P_-({\bf x})$, the Lebesgue measure on
 $\partial C_- \cong \mathbb{R} / (z \to z+T)$ corresponds to the quotient of the natural extension measure $\widehat{\ell}$ on the space of bi-infinite orbits.
 \end{lemma}
 
 \begin{proof}
(i) Let $A \subset \mathbb{R}$ be a wandering set of Lebesgue measure at least $T - \varepsilon$, provided by Lemma \ref{T-minus-epsilon}. By Lemma \ref{wandering-equality}, $\ell \bigl (P_-(\widehat{A} \cap \LS) \bigr ) \ge T - \varepsilon$.
Since $\partial C_-$ has Lebesgue measure $T$, letting $\varepsilon \to 0$ shows that $P_-(\LS)$
covers $\partial C_-$ up to a null set.

(ii) follows directly from (i).

(iii)  Since every measurable subset of $\mathbb{R}$ can be represented as a countable union of wandering sets (up to a null set), wandering sets generate the Borel $\sigma$-algebra of $\mathbb{R}$. It follows that sets $\widehat{A}$ with $A$ wandering, together with their images under $\widehat{F}^{\circ n}$, $n \in \mathbb{Z}$, generate the Borel sigma algebra of $\widehat{\mathbb{R}}$. Since $\ell \bigl (P_-(\widehat{A} \cap \LS) \bigr ) = \widehat{\ell} \bigl (\widehat{A} \cap \LS \bigr )$ for any wandering set $A$ by Lemma \ref{wandering-equality}, the map
$P_- : (\widehat{\mathbb{R}}, \widehat{\ell})/\langle \widehat{F} \rangle \to (\partial C_-, \ell)$ is measure-preserving. By (i), $P_-$ is a measure space isomorphism.
 \end{proof}

\begin{corollary}
\label{equality-of-lyapunov-exponents}
For any function $\phi \in L^1(\mathbb{R}, \ell)$, we have
$$\int_{\widehat{\mathbb{R}} / \langle \widehat{F} \rangle} \biggl \{  \sum_{n=-\infty}^\infty \phi(x_n) \biggr \} \, d\widehat{\ell} ({\bf x}) = \int_{\mathbb{R}} \phi (x) d\ell (x).$$
In particular,
$$\int_{\partial C_-} \log |(F^\infty)'| d\ell = \int_{\mathbb{R}} \log |F'| d\ell.$$
\end{corollary}

\subsection{Applications to Lavaurs maps and semigroups}

Recall that the Lavaurs map $L_\sigma$ was defined as $F^\infty \circ \tau_\sigma$ where $\tau_\sigma: C_+ \to C_-$ is  translation by $\sigma$ modulo $T$. According to Lemma \ref{c-minus-correspondence}(ii) and Corollary \ref{equality-of-lyapunov-exponents}, $L_\sigma$ is an inner function with
$$\int_{\partial \mathbb{D}} \log |L_\sigma'| dm = \int_{\mathbb{R}} \log |F'| d\ell.$$

To see that the Poggi-Corradini conjugacy $\Psi$ is an inner function, note that by Lemma \ref{c-minus-correspondence}, up to sets of measure zero, $\Psi$ identifies sets of the form $\{ x + n T : n \in \mathbb{Z}\}$ with bi-infinite orbits in $\widehat{\mathbb{R}}$.

By the remark (iii) at the end of Section \ref{sec:canonical-conjugacies}, the Pommerenke conjugacy $\widetilde{P}_+: \mathbb{H} \to \mathbb{H}$ is an inner function and hence so is the induced map $\widetilde{P}_+: \mathbb{H} \to C_+$. 
From the relation $P_+ \circ G = L \circ P_+$, it follows that
$G: \mathbb{H} \to \mathbb{H}$ is inner. (Indeed, if $G$ were not inner, then $P_+ \circ G$ would also not be inner, whereas $L \circ P_+$ is inner being the composition of two inner functions.)
 
\subsection{Ergodicity}

A measure $\mu$ on a space $X$ is {\em ergodic} with respect to a semigroup if every measurable set $E \subset X$ which is invariant under all maps in the semigroup either has zero measure or full measure.

\begin{lemma}
\label{lavarus-ergodicity}
The Lebesgue measure on the real line $\ell$ is ergodic with respect to the Lavaurs semigroup $\mathcal G$.
\end{lemma}

\begin{proof}
Let $E \subset \mathbb{R}$ be a measurable set that is invariant under the Lavaurs semigroup $\mathcal G$ and $u = P[\chi_E]$ be the Poisson extension of its characteristic function to the upper half-plane. Since $u(\phi(z))$, $\phi \in \mathcal G$ are bounded harmonic functions with the same boundary values, they are equal: $u(z) = u(\phi(z))$ for all $\phi \in \mathcal G$.

Since $u$ is $F$-invariant, it descends to a harmonic function $v$ on the forward cylinder $C_+$. The $G$-invariance of $u$ translates to the $L$-invariance of $v$\,:
$$
v(L(w)) = v(w), \qquad w \in C_+.
$$
As $v$ is a bounded harmonic function on $C_+ \cong \mathbb{D}^*$, it extends to a harmonic function on the whole unit disk, where it continues to satisfy $v(L(w)) = v(w)$. Since $L$ has an attracting fixed point at the origin, $v$ is constant on $\mathbb{D}$ and hence $u$ is constant on $\mathbb{H}$. Recalling the definition $u = P[\chi_E]$, we see that $E$ has either zero or full measure. 
\end{proof}

\begin{remark}
 If $G$ were doubly parabolic, then $\ell$ would already be ergodic with respect to the single map $G$, see Table \ref{tab:adm}. Unfortunately, we only know that $G$ has infinite height. (Since $G$ may not be a Polya-Szeg\"o inner function, we cannot use Table \ref{tab:adm2} to conclude that $G$ is doubly parabolic and that $\ell$ is ergodic with respect to $G$.)
\end{remark}

\section{Riemann surface laminations}
\label{sec:laminations}

In this section, we  review the theory of Riemann surface laminations for doubly parabolic case inner functions with finite Lyapunov exponent from \cite{laminations}. These constructions and results also hold for hyperbolic inner functions with minimal changes. We then associate a Riemann surface lamination $\widehat{X}_{\mathcal G}$ to a Lavaurs-Epstein semigroup $\mathcal G$. Rather than developing the theory of Riemann surface laminations from scratch for $\widehat{X}_{\mathcal G}$, we show that $\widehat{X}_{\mathcal G} \cong \widehat{X}_{L}$ and then transfer the volume form and geodesic flow from $\widehat{X}_{L}$ to $\widehat{X}_{\mathcal G}$. This allows the reader to take most of the inner workings of \cite{laminations} for granted.

\subsection{The doubly parabolic case}
\label{sec:doubly-parabolic-case}

Let $F: \mathbb{H} \to \mathbb{H}$ be a doubly parabolic inner function with the Denjoy-Wolff point at infinity whose Lyapunov exponent $\int_{\mathbb{R}} \log |F'| d\ell < \infty$. Form the space of bi-infinite orbits
$$
\widehat{\mathbb{H}} = \bigl \{ (z_n)_{n=-\infty}^\infty: z_n \in \mathbb{H}, \ z_{n+1} = F(z_{n}) \bigr \}.
$$
As in hyperbolic case, the Riemann surface lamination $\widehat{X}_F$ is obtained by taking the quotient $\widehat{\mathbb{H}}/\langle \widehat{F} \rangle$ by the map $\widehat{F}$ which applies $F$ to each coordinate. 
We also consider the {\em solenoid}
$$
\widehat{\mathbb{R}} = \bigl \{ (\zeta_n)_{n=-\infty}^\infty: \zeta_n \in \mathbb{R}, \ \zeta_{n+1} = F(\zeta_{n}) \bigr \}.
$$

\subsubsection{The volume form}

For a set $E \subset \mathbb{H}$, we write $\widehat{E} \subset \widehat{\mathbb{H}}$ for the set of bi-infinite orbits ${\bf z} = (z_n)_{n=-\infty}^\infty$ with $z_0 \in E$.
In \cite[Section 11]{laminations}, it is proved that
\begin{equation}
\label{eq:xi-def-dp}
\xi(\widehat{E}) = \lim_{n \to \infty} \int_{F^{-n}(E)} \frac{dA(z)}{\im z}
\end{equation}
defines an $\widehat{F}$-invariant measure on $\widehat{\mathbb{H}}$ which descends to a measure (also denoted $\xi$) on $\widehat{X}_F$ with total mass
$\xi(\widehat{X}_F) = \int_{\mathbb{R}} \log |F'| d\ell$.

Meanwhile, the solenoid possesses the measure $\widehat{\ell}$, which is the unique $\widehat{F}$-invariant measure on $\widehat{\mathbb{R}}$ which projects to $\ell$ under every coordinate.

\subsubsection{Rescaling limits}

For two points $w_1, w_2 \in \mathbb{H}$, we write $$L_{w_1 \to w_2}(z) = Az + B, \qquad A > 0, \quad B \in \mathbb{R},$$ for the unique element of $\aut(\mathbb{H}, \infty)$ which takes $w_1$ to $w_2$.

We say that $F$ has a (linear) rescaling limit along an inverse orbit
$
{\bf z} = (z_{-n})_{n=0}^\infty
$
if the sequence of maps
$$
F^{\circ m} \circ L_{i \to z_{-m}}, \qquad m = 0, 1, 2, \dots,
$$
converges uniformly on compact subsets of $\mathbb{H}$. If the limit exists, we denote it by $F^{\infty}_{{\bf z}} = F^{\infty}_{{\bf z}, 0}$ and say that backward iteration is asymptotically linear along ${\bf z}$. Evidently, if $F^{\infty}_{{\bf z}, 0}$ exists, then so does
$$
F_{{\bf z}, -n} = \lim_{m \to \infty} F^{\circ m} \circ L_{i \to z_{-n-m}}
$$
for any $n \ge 0$.

We define the {\em linear lamination} $\widehat{\mathbb{H}}_{\lin}$ as the set of inverse orbits along which $F$ has a rescaling limit.
According to \cite[Theorem 12.1]{laminations}, $\xi(\widehat{\mathbb{H}} \setminus \widehat{\mathbb{H}}_{\lin})$, i.e.~$F$ is asymptotically linear along $\xi$ a.e.~inverse orbit ${\bf z} \in \widehat{\mathbb{H}}$. 

\subsubsection{The geodesic flow}
\label{sec:2p-geodesic-flow}

The {\em geodesic flow} $\{ g_t: t \in \mathbb{R}\}$ acts on $\widehat{\mathbb{H}}_{\lin}$ by
$$
g_t({\bf z})_{-m} = F_{{\bf z}, -m}(e^t \cdot i).
$$
Since the geodesic flow on $\widehat{\mathbb{H}}_{\lin}$ is $\widehat{F}$-invariant, it descends to a geodesic flow defined
on a full measure subset $\widehat{X}_{F, \lin} \subset \widehat{X}_F$. Below, we will abuse notation and say that the flow acts on $\widehat{X}_F$.

We recall three properties of backward trajectories from \cite[Theorem 13.1]{laminations}:
\begin{enumerate}
\item For $\xi$ a.e.~${\bf z} = (z_{-n})_{n=0}^\infty \in \widehat{\mathbb{H}}$, the backward trajectory
$g_{-t}({\bf z})$
lands on an inverse orbit ${\bf x} = \zeta({\bf z}) = (x_{-n})_{n=0}^\infty \in \widehat{\mathbb{R}}$ on the real line:
$$
\lim_{t \to \infty} g_{-t}({\bf z})_{-n} = x_{-n}, \qquad n \ge 0.
$$

\item Let $\overline{\gamma}(t) \subset \mathbb{H}$ be the vertical geodesic that lands on $x_0 = \zeta_0({\bf z})$, parametrized by unit hyperbolic speed. The path $\gamma(t) = (g_{-t} {\bf z})_0$ is weakly shadowed by
$\overline{\gamma}(t)$\,: 
there exist an offset $t_0 \in \mathbb{R}$ and an increasing absolutely continuous reparametrization $\tau: [0, \infty) \to [0, \infty)$ satisfying
$$
\frac{1}{T}\int_0^T |\tau'(t)-1|\,dt \to 0, \qquad \text{as }T \to \infty,
$$
such that for every $\varepsilon > 0$,
$$ 
\frac{1}{T} \,
 \Bigr | \Bigl\{
t\in[0,T]:
d_{\mathbb H}\bigl(\gamma(t),\overline{\gamma}(t_0+\tau(t))\bigr)>\varepsilon
\Bigr\} \Bigr |
\to 0, \qquad \text{as }T \to \infty.
$$
Intuitively, this says that $\gamma$ stays close to $\overline{\gamma}$ on average and has the same large scale geometry.

\item The map $\zeta: (\widehat{\mathbb{H}}, \xi) \to (\widehat{\mathbb{R}}, \widehat{\ell})$ is non-singular in the sense that the pre-image of any measure zero set is a measure zero set.
\end{enumerate}

With help of Properties 1 and 3, one can deduce the ergodicity of the geodesic flow of $g_t: \widehat{X} \to  \widehat{X}$ from the ergodicity of $F: \mathbb{R} \to \mathbb{R}$, see \cite[Corollary 13.2]{laminations} for details. 

\subsubsection{Almost invariant functions}

 We say that a function $h: \mathbb{H} \to \mathbb{C}$ is {\em (weakly) almost invariant} under $F$ if for $\xi$ a.e.~backward orbit ${\bf z} = (z_{-n})_{n=0}^\infty \in \widehat{\mathbb{H}}$, the limit
 $$
\widehat{h}({\bf z}) := \lim_{n \to \infty} h(z_{-n})
$$
exists. We refer to $\widehat{h}$ as the {\em natural extension} of $h$ to $\widehat{X}_F$. 
We now recall \cite[Theorem 4.1]{laminations} which generalizes \cite[Theorem 10.6]{mcmullen} from finite Blaschke products to inner functions and slightly weakens the invariance assumption on $h$\,:

\begin{theorem}
\label{ergodic-theorem}
Let  $F: \mathbb{H} \to \mathbb{H}$ be a doubly parabolic inner function with finite Lyapunov exponent. If
 $h: \mathbb{H} \to \mathbb{C}$ is a bounded almost invariant function that is uniformly continuous in the hyperbolic metric, then for almost every $x \in \mathbb{R}$, we have
$$
\lim_{t \to 0} \frac{1}{|\log t|} \int_t^1 h(x+iy) \cdot \frac{dy}{y}  = \frac{1}{\int_{\mathbb{R}} \log|F'| d\ell} \int_{\widehat X} \widehat{h} d\xi.
$$
In particular, for any bounded interval $I \subset \mathbb{R}$,
$$
\lim_{t \to 0} \frac{1}{|\log t|} \int_I \int_t^1 h(x+iy) \cdot \frac{dx dy}{y}  = \frac{|I|}{\int_{\mathbb{R}} \log|F'| d\ell} \int_{\widehat X} \widehat{h} d\xi.
$$
\end{theorem}

The proof uses all three properties of backward trajectories above: By ergodicity of the geodesic flow, for a generic backward trajectory $\gamma(t)$,
the average of $\widehat{h}$ along $\gamma([0,T])$  converges to $$\frac{1}{\int_{\mathbb{R}} \log|F'| d\ell} \int_{\widehat X} \widehat{h} d\xi.$$ Properties 1 and 2 together with the uniform continuity of $h$ imply that the average of $\widehat{h}$ along $\gamma([0,T])$ is close to the average of $h$
along the vertical geodesic $\overline{\gamma}([0,T])$. Finally, Property 3 allows one to conclude that the endpoints of the resulting vertical geodesics are distributed with respect to the Lebesgue measure.

\subsubsection{Orbit Counting up to a Ces\`aro average}

For doubly parabolic inner functions, the Orbit Counting Theorem up to a Ces\`aro average reads as follows:

\begin{theorem}
\label{main-thm4c}
Suppose $F: \mathbb{H} \to \mathbb{H}$ is a doubly parabolic inner function with finite Lyapunov exponent $\int_{\mathbb{R}} \log |F'| d\ell < \infty$.
For a bounded interval $I$ in the real line, let
$$
\mathcal N_I(z, R) = \# \bigl \{ w \in I \times [e^{-R}, 1] : F^{\circ n}(w) = z \text{ for some }n \ge 0 \bigr \}.
$$
 For $z \in \mathbb{H}$ outside a set of zero measure, we have
$$
\frac{1}{R} \int_0^R \frac{\mathcal N_I(z, S)}{e^S} dS \, \sim \,  |I| \cdot \frac{1}{\int_{\mathbb{R}} \log |F'| d\ell}
$$
as $R \to \infty$. 
\end{theorem}

As explained in \cite[Sections 4 and 14]{laminations}, to count the number of repeated pre-images $w$ of a point $z \in \mathbb{H}$ with $\im w > e^{-R}$, one picks small parameters $\delta, \varepsilon, \eta > 0$ and applies Theorem \ref{ergodic-theorem} to an appropriate almost invariant function $h = h_{z, \delta, \varepsilon, \eta}$ and then takes $\delta, \varepsilon, \eta \to 0$ appropriately. In this overview, we will limit only to informally describing the almost invariant function $h$ used. The function $h$ is concentrated on the box
$$
\square_z = \{ z' \in \mathbb{H} \, : \, z - \delta < \re z' < z + \delta, \, (1 - \delta) \im z < \im z' < (1 + \delta) \im z \}
$$
and its  $\varepsilon$-linear inverse images.
The construction is as follows:

\begin{enumerate}
\item Let $w$ be a repeated pre-image of $z$, i.e.~$F^{\circ n}(w) = z$ for some $n \ge 0$. We say that $w$ is a {\em good} pre-image of $z$ if $F^{-n}$ admits an $\varepsilon$-linear inverse branch on $\square_z$ which takes $w$ to $z$, in which case,
$\square_w := F^{-n}(\square_z)$ is a topological box of roughly the same hyperbolic size and shape as $\square_z$.
We define $h_{\rough}(w') = 1$ if $w'$ belongs to a good box $\square_w$, and $h_{\rough}(w') = 0$ otherwise.

\item We now smoothen the function from the previous step. To that end, consider a slightly smaller box $\square_z(\delta - \eta)$ with $\eta <\!\!< \delta$. Define $h$ to be a smooth function on $\square_z$ which is $1$ on $\square_z(\delta - \eta)$, $0$ on $\partial \square$, and takes values between 0 and 1. Extend $h$ to the good boxes $\square_w$ by backward invariance. Finally, set $h  = 0$  on the rest of the upper half-plane. From the Schwarz lemma, it follows that $h$ is uniformly continuous in the hyperbolic metric.
\end{enumerate}

\subsubsection{Orbit Counting for one component inner functions}
\label{sec:oc-p1c}

We now examine Riemann surface laminations of doubly parabolic one component inner functions with finite Lyapunov exponent.
By \cite[Theorem 9.1]{inner-tdf}, a parabolic inner function is a parabolic one component inner function if and only if it is a covering map over an infinite strip $\{ -\rho < \im z < \rho \}$ for some $\rho > 0$. Together with Koebe's distortion theorem, this assumption implies that backward trajectories enjoy a stronger shadowing property which strengthens Property 2 from Section \ref{sec:2p-geodesic-flow}:

\begin{enumerate}
\item[$2'$.]The trajectory $\gamma(t) = (g_{-t} {\bf z})_0$ satisfies
$$
d_{\mathbb{H}} ( \gamma(t), \, \overline{\gamma}(t_0 + t))  \to 0, \qquad \text{as }T \to \infty,
$$
for some offset $t_0 \in \mathbb{R}$ depending on ${\bf z}$.
\end{enumerate}

This strong shadowing property plays an important role in the proofs of the following assertions, which we record as a single theorem:

\begin{theorem}
\label{parabolic-mixing-theorem}
{\em (i)} The geodesic flow $g_t: \widehat{X}_F \to \widehat{X}_F$ is mixing.

{\em (ii)} If $h: \mathbb{H} \to \mathbb{C}$ is a bounded almost invariant function that is uniformly continuous in the hyperbolic metric, then
$$
\lim_{y \to 0} \int_{I} h(x+iy) d\ell(x) = \frac{|I|}{\int_{\mathbb{R}} \log|F'| d\ell} \int_{\widehat X} \widehat{h} d\xi.
$$

{\em (iii)}
$$
\mathcal N_I(z, R) \, \sim \,  |I| \cdot \frac{e^R}{\int_{\mathbb{R}} \log |F'| d\ell}, \qquad \text{as }R \to \infty.
$$
\end{theorem}

\smallskip

See  \cite[Lemma 5.6, Theorem 4.2 and Theorem 1.2]{laminations} respectively.

\subsection{Lavaurs semigroups}

Let $F$ be a simply parabolic inner function with finite Lyapunov exponent and $G = G_\sigma$ be a commuting holomorphic map. Form the space of extended bi-infinite orbits
$$
\widehat{\mathbb{H}}_{\mathcal G} = \bigl \{ (z_{n, k})_{n, k =-\infty}^\infty: z_{n, k} \in \mathbb{H}, \ z_{n+1, k} = F(z_{n,k}), \ z_{n, k+1} = G(z_{n,k}) \bigr \},
$$
and define the Riemann surface lamination as the quotient $\widehat{X}_{\mathcal G} := \widehat{\mathbb{H}}/\mathcal G$, where two bi-infinite orbits ${\bf z} = (z_{n, k})_{n=-\infty}^\infty$ and ${\bf z'} = (z_{n, k})_{n=-\infty}^\infty$ are equivalent if there exist integers $n_0, k_0 \in \mathbb{Z}$ such that $z_{n, k} = z_{n+n_0,k+k_0}$ for all $n, k \in \mathbb{Z}^2$. We will also be interested in 
the intermediate space
$
\widehat{\mathbb{H}}_{\mathcal G} / \langle \widehat{F} \rangle
$
where ${\bf z} \sim {\bf z'}$ if there exists an integer $n_0 \in \mathbb{Z}$ such that $z_{n, k} = z'_{n+n_0,k}$ for all $n, k \in \mathbb{Z}^2$.
Finally, we define the extended solenoid
$$
\widehat{\mathbb{R}}_{\mathcal G} = \bigl \{ (z_{n, k})_{n, k =-\infty}^\infty: z_{n, k} \in \mathbb{R}, \ z_{n+1, k} = F(z_{n,k}), \ z_{n, k+1} = G(z_{n,k}) \bigr \},
$$
which fibers over the usual solenoid
$$
\widehat{\mathbb{R}}_{F} = \bigl \{ (z_{n})_{-\infty}^\infty: z_{n} \in \mathbb{R}, \ z_{n+1} = F(z_{n+1}) \bigr \},
$$
considered previously. The extended solenoid carries the natural extension measure $\widehat{\ell}_{\mathcal G}$, which is the unique measure that is invariant under both $\widehat{F}$ and $\widehat{G}$ and projects to the Lebesgue measure under every coordinate.

As any bi-infinite orbit ${\bf z} \in \widehat{\mathbb{H}}_{\mathcal G}$ ``passes'' through the forward half-cylinder $C^+$ bi-infinitely many times, it induces a bi-infinite orbit 
$${\bf w} = \widehat{P}_+({\bf z}), \qquad w_n = P_+(z_{n,k}),$$
 of the Lavaurs map $L = L_\sigma: C_+ \to C_+$. Completing the half-cylinder to a disk by adding a point at infinity, 
we may view $\widehat{P}_+({\bf z})$ as a bi-infinite orbit in $\widehat{\mathbb{D}}_{L}$. From the construction, it is clear that $\widehat{P}_+$ descends to bijections $$\widehat{\mathbb{H}}_{\mathcal G}/ \langle \widehat{F} \rangle \to \widehat{\mathbb{D}}_{L} \qquad \text{and} \qquad \widehat{X}_{\mathcal G} \to \widehat{X}_{L}.$$

Since the Lavaurs map $L$ is a hyperbolic inner function (with a Denjoy-Wolff point at the origin), the orbit space $\widehat{\mathbb{D}}_{L}$ carries a natural volume form $\xi_{L}$ and a geodesic flow $g_{t, L}$.
We define a volume form $\xi_{\mathcal G}$ and a geodesic flow $g_{t, \mathcal G}$ on $\widehat{\mathbb{H}}_{\mathcal G}$ by transferring these structures from $\widehat{\mathbb{D}}_{L}$. In view of Theorem \ref{summary2}(ii), the total volume of $\widehat{X}_{\mathcal G}$ is
$$
\xi_{\mathcal G}(\widehat{X}_{\mathcal G}) \, = \, \xi_{L}(\widehat{X}_L) \, = \, \int_{\partial \mathbb{D}} \log |L'| dm \, = \, \int_{\mathbb{R}} \log |F'| d\ell.
$$
Since $g_{t, L}$ on $\widehat{X}_L$ is ergodic, so is $g_{t, \mathcal G}$ on $\widehat{X}_{\mathcal G}$.

Recall from Lemma \ref{c-minus-correspondence}(iii) that $P_-:  (\widehat{\mathbb{R}}_F, \widehat{\ell}_F ) / \langle \widehat{F} \rangle \to (\partial C_-,  \ell)$ is a measure space isomorphism. Since $F^\infty: C_- \to C_+$ is an inner function by Lemma \ref{c-minus-correspondence}(ii) and $P_+ = F^{\infty} \circ P_-$, L\"owner's lemma implies that $P_+: (\widehat{\mathbb{R}}_F, \widehat{\ell}_F) / \langle \widehat{F} \rangle \to (\partial C_+, \ell)$ is also a measure space isomorphism.

Let $\widehat{S}^1_L$ be the solenoid associated to the Lavaurs map $L: C_+ \to C_+$.  The solenoid $\widehat{S}^1_L$  carries the natural extension measure $\widehat{\ell}_L$, the unique $L$-invariant measure on ${\widehat S}^1_{L}$ which projects to the Lebesgue measure under every coordinate. It is readily seen that 
$P_+:  (\widehat{\mathbb{R}}_{\mathcal G}, \widehat{\ell}_{\mathcal G} )  / \langle \widehat{F} \rangle \to (\widehat{S}^1_{L}, \widehat{\ell}_L)$ is a measure space isomorphism.

 \begin{lemma}
 \label{backard-trajectories-for-lavaurs-semigroups}
Backward geodesic trajectories in $\widehat{\mathbb H}_{\mathcal G}$ enjoy Properties 1--3 from Section \ref{sec:2p-geodesic-flow}.
Namely, they land at points in the solenoid $\widehat{\mathbb R}_{\mathcal G}$, are weakly shadowed by vertical geodesics, and the endpoint mapping $\zeta: (\widehat{\mathbb H}_{\mathcal G}, \xi_{\mathcal G}) \to (\widehat{\mathbb{R}}, \widehat{\ell}_{\mathcal G})$ is non-singular.
 \end{lemma}

\begin{proof}
From the background material in Section \ref{sec:doubly-parabolic-case}, we know that Properties 1--3 are satisfied for backward  trajectories of the geodesic flow in $\widehat{\mathbb{D}}_L$. The lemma says that these properties transfer to backward trajectories in 
$\widehat{\mathbb{H}}_{\mathcal G}$.

We know that for $\xi_L$ a.e.~${\bf w} \in \widehat{\mathbb{D}}_L$, for any $k \in \mathbb{Z}$, the path $g_{-t}({\bf w})_k$ lands at a point $\zeta_{k} \in \partial C_+$ and is weakly shadowed by the vertical geodesic with endpoint $\zeta_k$.

Applying the map $\tau_\sigma$, we see that the path
$\tau_\sigma(g_{-t}({\bf w})_k) \subset C_-$ lands at $\tau_\sigma(\zeta_k) \in \partial C_-$ and is weakly shadowed by the vertical geodesic with endpoint $\tau_\sigma(\zeta_k)$.

We consider all lifts of the path $\tau_\sigma(g_{-t}({\bf w})_k)$ under the projection map $z \to z \!\!\mod T$ from $\mathbb{H} \to C_-$. Thus, we obtain infinitely many paths $\tau_{\sigma,n}(g_{-t}({\bf w})_k)$, $n \in \mathbb{Z}$, in the upper half-plane that land are weakly shadowed by vertical geodesics.

Applying the map $\Psi$, we obtain the paths
$g_{-t}({\bf z})_{n, k+1}$, $n \in \mathbb{Z}$, where ${\bf z} \in \widehat{\mathbb{H}}_{\mathcal G}$ projects onto ${\bf w} \in \widehat{\mathbb{D}}_L$ under $\widehat{P}_+$. 

Since $\Psi$ has an angular derivative $\ell$ a.e.~on $\mathbb{R}$, after removing a $\xi_{\mathcal G}$ measure zero set of inverse orbits ${\bf w} \in \widehat{\mathbb{D}}_L$, we may assume that the paths $\tau_{\sigma,n}(g_{-t}({\bf w})_k)$, $n \in \mathbb{Z}$, land at points on the real line where $\Psi$ has an angular derivative.
Consequently, the paths $g_{-t}({\bf z})_{n, k+1} = \Psi \bigl (\tau_{\sigma,n}(g_{-t}({\bf w})_k) \bigr )$ land on the real line and are weakly shadowed by vertical geodesics.

From the compatibility of the measures $\xi_{\mathcal G}$ and $\xi_L$, it follows that for $\xi_{\mathcal G}$ a.e.~${\bf z} = (z_{n, k})_{n, k=0}^\infty \in \widehat{\mathbb{H}}_{\mathcal G}$, the backward trajectory
$g_{-t}({\bf z})$ lands on an extended bi-infinite orbit $\zeta({\bf z}) \in \widehat{\mathbb{R}}_{\mathcal G}$, and for each $n, k \in \mathbb{Z}$, $g_{-t}({\bf z})_{n,k}$ is weakly shadowed by a vertical geodesic in $\mathbb{H}$. Finally, in view of the compatibility of the measures $\widehat{\ell}_{\mathcal G}$ and $\widehat{\ell}_L$, the non-singularity of the map $\zeta_{\mathcal G}: (\widehat{\mathbb H}_{\mathcal G}, \xi_{\mathcal G}) \to (\widehat{\mathbb{R}}, \widehat{\ell}_{\mathcal G})$ follows from the non-singularity of the map $\zeta_{L}: (\widehat{\mathbb D}_{L}, \xi_{L}) \to (\widehat{S}^1_L, \widehat{\ell}_{L})$.
\end{proof}

With help of the above lemma, one can show that Theorem \ref{ergodic-theorem} holds in the simply parabolic setting. From here, the proof of the Orbit Counting Theorem up to a Ces\`aro average (Theorem \ref{main-thm3c}) follows the same general strategy as in the doubly parabolic case. The only difference is that when designing the almost invariant function $h$, one uses extended iteration instead of regular iteration.

\section{Parabolic one component inner functions}
\label{sec:p1c}

Recall from the introduction that an inner function $f: \mathbb{D} \to \mathbb{D}$ is a classical one component inner function if it is a covering map over an annulus $\{ r < |z| < 1 \}$ for some $0 < r < 1$, while a parabolic inner function $F: 
\mathbb{H} \to \mathbb{H}$ is a parabolic one component inner function if and only if it is a covering map over an infinite strip $\{ 0 < \im z < \rho \}$ for some $\rho > 0$. As explained in \cite[Theorems 7.2 and 9.1]{inner-tdf}, these characterizations remain valid if one replaces these domains by symmetric annuli
$\{ r < |z| < 1/r \}$ and strips $\{ -\rho < \im z < \rho \}$.

In this section, we study Lavaurs semigroups associated to simply parabolic one component inner functions with finite Lyapunov exponent. We show that in this setting, the Lavaurs map $L = L_\sigma$ is a classical one component inner function (Lemma \ref{p1c-structure}). As an application, we derive the full Orbit Counting Theorem for Lavaurs semigroups (Theorem \ref{main-thm3d}).

\begin{proof}[Proof of Lemma \ref{p1c-structure}]
(i) Suppose $F$ is a parabolic one component inner function with finite Lyapunov exponent. To show that the measure $\mu$ in its Polya-Szeg\"o representation is compactly supported, it is enough to check that $F'(x) \to 1$ as $x \to \pm \infty$.

From Koebe's distortion theorem and the covering property of parabolic one component inner functions, it follows that if $F'(x_0) > c > 1$ for some $x_0 \in \mathbb{R}$, then $F'(x) > \frac{1+c}{2}$ on an interval $I(x_0, r) = \{ x \in \mathbb{R} : | x - x_0| < r\}$, where $r = r(c, \rho) > 0$ depends only on $c$ and $\rho$.
Therefore, if $F$ has a finite Lyapunov exponent, then the set of $x_0 \in \mathbb{R}$ with $F'(x_0) > c$ must be bounded and so $F'(x) \to 1$ as $x \to \pm \infty$ as desired.

 Since $\mu$ is compactly supported, Lemma \ref{properties-of-Pplus} admits a stronger conclusion: the Pommerenke conjugacy 
$\widetilde{P}_+(z) = \lim_{n \to \infty} F^{\circ n}(z) -  \re F^{\circ n}(i)$
 injective on $$\Omega_+ = \{ x + iy \in \mathbb{H} : x > X \}.$$ Similarly, Lemma \ref{properties-of-Pminus}  strengthens to the injectivity of $\widetilde{P}_-(z) = \lim_{n \to \infty} \bigl ( F^{-n}(z) - \re c_{-n} \bigr )$
on $$\Omega_- = \{ x + iy \in \mathbb{H} : x < - X \}.$$

(ii) Let $F$ be a simply parabolic inner function with $T > 0$ and $\sigma \in \mathbb{R} / T\mathbb{Z}$. We show that if $F$ is a covering map over  $\{z \in \mathbb{H}: 0 < \im z < \rho \}$, then the Lavaurs map $L: C_+ \to C_+$ is a covering map over $\{ z \in C_+ : 0 < \im z < \rho \}$.

For this purpose, we fix a point $w_1 \in C_+$ with $\im w_1 < \rho$ and a pre-image $w_0 \in L^{-1}(w_1)$. Let $U = B_{\hyp}(w_1, \delta) \subset \{ z \in C_+ : \im z < \rho \}$ be a  small hyperbolic ball centered at $w_1$. (Here, $\delta > 0$ is chosen sufficiently small so that any hyperbolic ball of radius $\delta$ centered at a point with imaginary part less than $\rho$ is embedded in $C_+$). We construct a holomorphic branch of $L^{-1}$ on $U$ which takes $w_1$ to $w_0$.

Since $L = F^{\infty} \circ \tau_\sigma$, there exists a bi-infinite orbit 
${\bf z} = (z_{n,0})_{n \in \mathbb{Z}}$ in $\mathbb{H}$ such that $P_-({\bf z}) = \tau_\sigma(w_0)$ and $P_+({\bf z}) = w_1$.
Choose $N > 0$ sufficiently large so that:
\begin{itemize}
\item For $n \ge N$, the points $z_{n, 0} \in \Omega_+$,
\item For $n \le {-N}$, the points $z_{n, 0} \in \Omega_-$.
\end{itemize}
By enlarging $N$ if necessary, we may assume that hyperbolic balls of radius $\delta$ centered at these orbit points will also be in $\Omega_+$ or $\Omega_-$.

Since $\widetilde{P}_+$ is injective on $\Omega_+$, for each $n \ge N$, there exists a unique inverse branch of $P_+$ sending 
$w_1$ to $z_{n, 0}$. Let $U_n$ be the connected component of $P_+^{-1}(U)$ containing $z_{n,0}$. Since $\widetilde{P}_+$ is a parabolic inner function, it increases the imaginary part, and hence each $U_n$, $n \ge N$, is contained in the strip $\{ z \in \mathbb{H} : \im z < \rho \}$.

Since $F$ acts as a covering map on the horizontal strip $\{ z \in \mathbb{H} : \im z < \rho \}$, there exists a branch of $F^{-1}$ defined on $U_N$ which takes $z_{N, 0}$ to $z_{N-1,0}$. We define $U_{N-1}$ as the image of $U_N$ under this inverse branch. As parabolic inner functions increase the imaginary part, $U_{N-1}$ is again contained in the strip  $\{ z \in \mathbb{H} : \im z < \rho \}$. Continuing inductively, we obtain a sequence of simply-connected domains $\{U_k\}_{k \in \mathbb{Z}}$ such that $F$ conformally maps $U_k$ to $U_{k+1}$ and sends $z_{k,0}$ to $z_{k+1,0}$.

By the Schwarz lemma, $U_{-N} \subset B_{\hyp}(z_{-n,0}, \delta) \subset \Omega_-$. Since $\widetilde{P}_-$ is injective on $\Omega_-$, the image $P_-(U_{-N}) \subset B_{\hyp}(\tau_\sigma(w_0), \delta) \subset C_-$ is a simply-connected domain which maps univalently to $U$ under $F^\infty$. Therefore, $V := \tau_\sigma^{-1}(P_-(U_{-N}))$ is the desired neighbourhood of $w_1$.
The proof is complete.
 \end{proof}

The proof of Orbit Counting Theorem for one component Lavaurs semigroups (Theorem \ref{main-thm3d}) largely mimics that of Theorem \ref{parabolic-mixing-theorem} which concerns doubly parabolic one component inner functions. 
By the hyperbolic analogue of Theorem \ref{parabolic-mixing-theorem}, the geodesic flow on $\widehat{X}_L$ is mixing and backward geodesic trajectories $g_{-t}({\bf w})_n$ enjoy the strong shadowing property $2'$ from Section \ref{sec:oc-p1c}.
An argument similar to the one in Lemma \ref{backard-trajectories-for-lavaurs-semigroups} shows that backward geodesic trajectories $g_{-t}({\bf z})_{n, k} \subset \mathbb{H}$ also enjoy the strong shadowing property. From here, one can prove Theorem \ref{main-thm3d} as in \cite{laminations}.

\appendix

\section{An example when $G$ is not inner}
\label{sec:bad-example}

We now give an example of a Polya-Szeg\"o inner function $F$ for which the commuting holomorphic self-map $G$ is not an inner self-map of the upper half-plane, that is, $\im F(x) = 0$ for a.e.~$x \in \mathbb{R}$. In this example, the Lavaurs map $L$ is also not inner.

\begin{lemma}
There exists a Polya-Szeg\"o function $F$ and $c > 0$ such that $\im G(z) > c$ for all $z \in \mathbb{H}$.
\end{lemma}

Our example is of the form
$$
F(z) = z + 1/3 - \int_{\mathbb{R}} \frac{d\mu(a)}{z-a},
$$
where
$$
\mu \, = \, \sum_{j=1}^\infty \mu_j \, = \, \sum_{j=1}^\infty \frac{M}{j^2} \cdot \nu_{-t_j, j},
\qquad
\nu_{-t, n} = \frac{1}{n} \sum_{k=0}^{n- 1} \delta_{- t + k/n}.
$$
In other words, the measure $\mu$ is composed of the infinitely many pieces $\mu_j$, where the piece $\mu_j$ is supported on the platform $I_j = [-2j, -2j + 1]$ of length 1. 

\begin{lemma}
\label{up-up-away}
For any $M > 0$ sufficiently large, there exists a $c > 0$ such that for any bi-infinite orbit $(z_n)_{n=-\infty}^\infty$ which satisfies
\begin{equation*}
\label{eq:up-up-away-BISBS}
z_{n+1} - z_n \to T, \qquad z_{n} \!\!\mod T \, \to \, w \, \in \, C_-, \qquad \text{as }n \to -\infty,
\end{equation*}
 we have
$\lim_{n \to \infty} \im z_n > c$.
\end{lemma}

The lemma implies that the image of $F^\infty: C_- \to C_+$ is contained in the set $$\bigl \{ w \in \mathbb{H} : \im F^\infty(w) > c \bigr \} \, \bigl / \, (z \to z + T) \, \subset \, C_+.$$
Since $L_\sigma = F^\infty \circ \tau_\sigma$, the same is true for the Lavaurs map $L_\sigma: C_+ \to C_+$. As $\im G(w) \ge \im P_+(w)$, we also have $\im G(w) > c$ for any $w \in \mathbb{H}$. Thus assuming the above lemma, $G$ is not an inner self-map of the upper half-plane.

We say that a bi-infinite orbit ${\bf z} = (z_n)_{n=-\infty}^\infty$ {\em passes above} the platform $I_j$ if there is an orbit point $z_n$ with $\re z_n \in I_j$ and $\im z_n > 1/j$.

\begin{proof}
We assume that $\im z_n < 1$ for all $n \in \mathbb{Z}$, otherwise there is little to show. We show that when $M$ is sufficiently large, $z_n$ passes above all platforms $I_j$ with $j \ge j_0$, where $j_0$ is to be determined.

Suppose $z_n$ is located above the platform $I_{j+1}$. If $M > 0$ is sufficiently large, then after one step, ${\bf z}$ has enough height to pass over $I_{j}$\,:
$$
\im z_{n+1} - \im z_n \gtrsim \, \mu(I_j) \, > \, 1/j^2 \, > \, 1/(j+1) - 1/j.
$$
Additionally, if $\im z_n > 1/(j+1)$ and $\re z_n \in [-2j-2, -2j + 1]$, then
$$
\re z_{n+1} - (1/3 + \re z_n) \lesssim M/j.
$$
When $j \ge j_0$ is sufficiently large, the right hand side is less than $1/3$. Therefore, if ${\bf z}$ passes above $I_{j+1}$, it will not skip $I_{j}$. Consequently, any bi-infinite orbit ${\bf z}$ which satisfies the hypotheses of the lemma passes above all platforms with indices $j \ge j_0$.
\end{proof}

\section{Aleksandrov-Clark measures}
\label{sec:ac-measures}

In this appendix, we review the classical definition of Aleksandrov-Clark measures for holomorphic self-maps of the unit disk. Afterwards, we consider an analogue of Aleksandrov-Clark measures for parabolic self-maps of the upper half-plane.

Let $f$ be a holomorphic self-map of the unit disk. 
For $\alpha \in \partial \mathbb{D}$, consider the function
$$
h_\alpha(z) = \frac{\alpha+f(z)}{\alpha-f(z)}, \qquad z \in \mathbb{D}.
$$
Since $\re h_\alpha$ is a positive harmonic function, it can be represented as the Poisson extension of a positive measure on the unit circle:
\begin{equation}
\label{eq:ac-def}
\re h_\alpha(z) = \re \int_{\partial \mathbb{D}} \frac{\zeta+z}{\zeta-z} \, d\mu_\alpha(\zeta).
\end{equation}
The collection of measures $\{ \mu_\alpha \}_{\alpha \in \partial \mathbb{D}}$ are known as the {\em Aleksandrov-Clark measures} of $f$. The measures $\mu_\alpha$ enjoy the following properties:

\begin{enumerate}
\item With the normalization $f(0) = 0$, each measure $\mu_\alpha$ has unit mass. When $f$ is an inner function, for each $\alpha \in \partial \mathbb{D}$, the function $\re h_\alpha$ has radial limit zero a.e.~on the unit circle, implying that $\mu_\alpha$ is singular. 

\item The singular part of the Aleksandrov-Clark measure $\mu_\alpha$ is concentrated on the set of points on the unit circle where $f$ has radial limit $\alpha$. Moreover, 
$\mu_\alpha$ has a point mass at $\zeta \in \partial \mathbb{D}$ if and only if $f$ has an angular derivative at $\zeta$ with $f(\zeta) = \alpha$, in which case, $$\mu_\alpha(\{\zeta\}) = |f'(\zeta)|^{-1}.$$

\item (Aleksandrov's Disintegration Theorem) For any $\phi \in L^1(\partial \mathbb{D}, m)$,
$$
\int_{\partial \mathbb{D}} \biggl ( \int_{\partial \mathbb{D}}  \phi(\zeta) d\mu_\alpha(\zeta) \biggr ) dm(\alpha) = \int_{\partial \mathbb{D}} \phi(\zeta)dm(\zeta).
$$
 \end{enumerate}
 For the proofs, we refer the reader to \cite{cima, PS, saksman}.

We now examine Aleksandrov-Clark measures for parabolic holomorphic self-maps $F: \mathbb{H} \to \mathbb{H}$ of the upper half-plane. For $\xi \in \mathbb{R}$, consider the function
$$
h_\xi(z) = \frac{i}{F(z)-\xi}, \qquad z \in \mathbb{H}.
$$
Since $\re h_\xi$ is a positive harmonic function on the upper half-plane, there exists a constant $c_\xi > 0 $ and a positive (but not necessarily finite) measure $\mu_\xi$ on $\mathbb{R}$ such that
$$
\re h_\xi(z) = c_\xi y + \frac{1}{\pi} \int_{\mathbb{R}} \frac{y}{x^2 + y^2} \, d\mu_\xi(x).
$$
As $F$ is parabolic with a Denjoy-Wolff point at infinity, $F(iy) = iy + o(y)$ as $y \to \infty$. In particular,
$\re h_\xi(z) = 1/y + o(1/y)$ as $y \to \infty$, which implies that $c_\xi = 0 $ and $\mu_\xi$ is a probability measure.
 
As in the unit disk setting, $\mu_\xi$ is singular with respect to Lebesgue measure and is supported on the set of points $\zeta \in \mathbb{R}$ where $F$ has vertical limit $\xi$. Furthermore,
$\mu_\xi$ has an atom at $\zeta \in \mathbb{R}$ if and only if $F$ has a finite angular derivative at $\zeta$ with $F(\zeta) = \xi$, in which case, $\mu_\xi(\{\zeta\}) = |F'(\zeta)|^{-1}$. These properties may be proved directly or deduced from the classical Aleksandrov-Clark theory on the unit disk via a M\"obius change of variables.
More precisely, if $f = m^{-1} \circ F \circ m$ where $m: \mathbb{D} \to \mathbb{H}$ is the M\"obius transformation which maps the unit disk to the upper half-plane and takes $p$ to $\infty$, then
$$
\mu_{\alpha} = |m'| \, d(m^* \mu_\xi), \qquad
\text{where }\xi = m(\alpha).
$$

We now show that a parabolic inner function $F$ is injective on the thin part of the real line:

\begin{proof}[Proof of Lemma \ref{thin-sets}]
Let $\{ \mu_\xi : \xi \in \mathbb{R}\}$ be the Aleksandrov-Clark measures associated to a parabolic inner function $F$. 
If two distinct points in $\mathbb{R}_{\thin} = \{ x \in \mathbb{R} : |F'(x)| < 2 \}$ were mapped to the same point $\xi$, then $\mu_\xi$ would have two atoms of mass greater than $1/2$, contradicting that $\mu_\xi$ is a probability measure.
\end{proof}

The upper half-plane analogue of Aleksandrov's disintegration theorem says that for any $\phi \in L^1(\mathbb{R}, \ell)$,
$$
\int_{\mathbb{R}} \biggl ( \int_{\mathbb{R}}  \phi(\zeta) d\mu_\xi(\zeta) \biggr ) d\ell(\xi) = \int_{\mathbb{R}} \phi(\xi)d\ell(\xi).
$$
We now deduce an upper half-plane analogue of L\"owner's lemma from this identity.
Suppose $A \subset \mathbb{R}$ has finite Lebesgue measure. Taking $\phi = \chi_{F^{-1}(A)}$, we get
$$
\int_A \mu_\xi  (F^{-1}(A) ) d\ell = \ell(F^{-1}(A)).
$$
Since each $\mu_\xi$ is a probability measure, $\ell(F^{-1}(A)) \le \ell(A)$.
For parabolic inner functions, each $\mu_\xi$ with $\xi \in A$ is supported on $F^{-1}(A)$ and so equality holds: $\ell(F^{-1}(A)) = \ell(A)$.

 \section{The exceptional set}
\label{sec:exceptional-set}

In this appendix, we discuss several characterizations of exceptional sets of inner functions on the unit disk. We then discuss an analogous result for parabolic inner functions acting on the upper half-plane.

 Let $f: \mathbb{D} \to \mathbb{D}$ be an inner function. A point $a \in \mathbb{D}$ is {\em exceptional} for $f$ if the Frostman shift
$$
f_a(z) = \frac{f(z)-a}{1-\overline{a} f(z)}
$$
fails to be a Blaschke product. Frostman showed that the exceptional set of an inner function has logarithmic capacity zero, e.g.~see \cite[Theorem 2.5]{mashreghi}. We write 
$$
G(z, w) = \log \biggl | \frac{1 - z\overline{w}}{z - w} \biggr | \qquad \text{and} \qquad P_z (\zeta) = \frac{1-|z|^2}{|z-\zeta|^2}
$$
 for the Green's function and the Poisson kernel of the unit disk respectively. 

\begin{lemma}
\label{exceptional1}
Suppose $f: \mathbb{D} \to \mathbb{D}$ is an inner function.
For any $a \in \mathbb{D}$,
\begin{equation}
\label{eq:exceptional1}
\sum_{f(b) = a} G(b, z) \le G(a, f(z)), \qquad z \in \mathbb{D} \setminus f^{-1}(a).
\end{equation}
Equality holds if and only if $z$ is not exceptional.
\end{lemma}

\begin{proof}
We factor $f$ into a singular inner function and Blaschke product:
$$
\exp \biggl ( - \int_{\partial \mathbb{D}} P_z(\zeta) d\sigma_f \biggr ) \, \prod_{f(b) = 0} \, \biggl | \frac{z-b}{1-\overline{b}z} \biggr | = |f(z)|,
$$
Assuming $f(z) \ne 0$, taking logarithms yields
$$
\int_{\partial \mathbb{D}} P_z(\zeta) d\sigma_f + \sum_{f(b) = 0} G(b, z) = G(0, f(z)),
 $$
which implies (\ref{eq:exceptional1}) for $a = 0$.
The general case follows by applying the above argument to the Frostman shift $f_a$ and using the M\"obius invariance of the Green's function.
\end{proof}

\begin{lemma}
Suppose $f$ has an angular derivative at $p \in \partial \mathbb{D}$. For any $a \in \mathbb{D}$,  
\begin{equation}
\label{eq:exceptional2}
\sum_{f(b) = a} P_b(p) \le |f'(p)| \cdot P_a(f(p)).
\end{equation}
Equality holds if and only if $a$ is not exceptional.
\end{lemma}

\begin{proof}
By \cite[Theorem 4.15]{mashreghi}, the angular derivative of an inner function $f$ at a point on the unit circle is given by
$$
\sum_{f(b) = 0} P_b(p) + 2 \int_{\partial \mathbb{D}} \frac{d\sigma_f(\zeta)}{|\zeta - p|^2} = |f'(p)|,
$$
which implies (\ref{eq:exceptional2}) for $a = 0$. The general case follows after applying the above identity with $f_a$ in place of $f$.
\end{proof}

When $f: \mathbb{D} \to \mathbb{D}$ has a parabolic Denjoy-Wolff point $p \in \partial \mathbb{D}$, the previous lemma says that 
$$
\sum_{f(b) = a} P_b(p) \le P_a(p),
$$
with equality if and only if $a \in \mathbb{D}$ is not exceptional. Let $m$ be M\"obius transformation which maps the unit disk to the upper half-plane and takes $p$ to $\infty$. We say that $z \in \mathbb{H}$ is an exceptional point for $F = m \circ f \circ m^{-1}$ if $m^{-1}(z) \in \mathbb{D}$ is an exceptional point for $f$.

\begin{corollary}
\label{exceptional3}
Suppose $F: \mathbb{H} \to \mathbb{H}$ is a parabolic inner function with a Denjoy-Wolff point at infinity. For any $z \in \mathbb{H}$, we have
$$
\sum_{F(w) = z} \im w \le \im z.
$$
Equality holds if and only if $z$ is not exceptional.
\end{corollary}

\bibliographystyle{amsplain}

\end{document}